\documentclass{amsart}

\usepackage{amsmath}
\usepackage{amsthm,amssymb,color,comment}
\usepackage{bbm}
\usepackage{xcolor}
\usepackage{hyperref}
\usepackage[TS1,T1]{fontenc}
\usepackage{inputenc}
\usepackage{dsfont}
\usepackage{tikz}
\usepackage{enumitem}
\usepackage{soul}
\usepackage[sort]{cite}% to arrange citations in increasing order

\numberwithin{equation}{section}

\newtheorem{thm}{Theorem}[section]
\newtheorem{prop}[thm]{Proposition}
\newtheorem{lem}[thm]{Lemma}
\newtheorem{cor}[thm]{Corollary}

\newtheorem{Def}[thm]{Definition}

\theoremstyle{definition}
\newtheorem{Ass}[thm]{Assumption}
\newtheorem{rem}[thm]{Remark}

\DeclareMathOperator{\DIV}{div}

\newcommand{\R}{\mathbb{R}}
\newcommand{\N}{\mathbb{N}}
\newcommand{\p}{\partial}

\newcommand{\diff}{\mathop{}\!\mathrm{d}}

\makeatletter
\newcommand{\doublewidetilde}[1]{{%
  \mathpalette\double@widetilde{#1}%
}}
\newcommand{\double@widetilde}[2]{%
  \sbox\z@{$\m@th#1\widetilde{#2}$}%
  \ht\z@=.9\ht\z@
  \widetilde{\box\z@}%
}
\makeatother
\author{Miroslav Bul\'{\i}\v{c}ek}
\address{Mathematical Institute, Faculty of Mathematics and Physics, Charles University, Sokolovsk\'{a} 83, 186~75, Prague, Czech Republic}
\email{mbul8060@karlin.mff.cuni.cz}
\thanks{Miroslav Bul{\'\i}{\v{c}}ek was supported by the project No. 20-11027X financed by GA\v{C}R}

\author{Jakub Woźnicki}
\address{Faculty of Mathematics, Informatics and Mechanics, University of Warsaw, Stefana Banacha 2, 02-097 Warsaw, Poland}
\email{jw.woznicki@student.uw.edu.pl}
\thanks{Jakub Woźnicki was supported by National Science Center, Poland through the project no. 2023/32/O/ST1/03031}

\begin{document}

\title[Parabolic equations with measure or integrable data]{Parabolic equations with non-standard growth and measure or integrable data}

\begin{abstract}
We consider a parabolic partial differential equation with Dirichlet boundary conditions and measure or $L^1$ data. The key difficulty consists in a presence of a monotone operator~$A$ subjected to a non-standard growth condition, controlled by the exponent $p$ depending on the time and the spatial variable. We show the existence of a weak and an entropy solution to our system, as well as the uniqueness of an entropy solution, under the assumption of boundedness and log-H\"{o}lder continuity of the variable exponent~$p$ with respect to the spatial variable. On the other hand, we do not assume any smoothness of~$p$ with respect to the time variable.
\end{abstract}


\maketitle

\section{Introduction}
In this work we deal with the following parabolic problem
\begin{align}\label{eq:mainequation}
    \p_t u(t,x) - \DIV_x A(t, x, \nabla_x u) = f(t,x),
\end{align}
which is completed by the following boundary and initial conditions
\begin{align}\label{eq:boundaryconditions}
u|_{\p\Omega} = 0, \quad u(0, x) = u_0(x).
\end{align}
Here, we assume that $\Omega\subset \mathbb{R}^d$ is a Lipschitz domain and $T$ is the length of time interval. We denote by $\Omega_T:=(0,T)\times \Omega$ the parabolic cylinder, and the equation~\eqref{eq:mainequation} is supposed to be satisfied in~$\Omega_T$. The operator $A$ is non-linear but elliptic, roughly having the form
\begin{equation}
A(t,x,\xi)\sim |\xi|^{p(t,x)-2}\xi.\label{prototypeA}
\end{equation}
We specify the precise assumption on the operator~$A$ in the next section. The relation~\eqref{prototypeA} just emphasizes that  the behavior of the operator $A$ is the $p(t,x)$-Laplacian-like setting.

Concerning the data, we assume that it is merely integrable, i.e., $f\in L^1(\Omega_T)$ and $u_0\in L^1(\Omega)$. Moreover, whenever we deal with very weak solutions, we will relax the assumption on integrability and assume that the data is just Radon measures, i.e., $f\in \mathcal{M}(\Omega_T)$ and $u_0\in \mathcal{M}(\Omega)$.

Our main goal is to introduce the well sounded existence and, if possible, also the uniqueness theory for a solution to~\eqref{eq:mainequation}. Let us make several comments about the motivation for the studied problem, and let us mention the main difficulties.

First, we consider the $p(t,x)$ setting such that
\begin{equation}\label{sta1}
1<p_{\min}\le p(t,x)\le p_{\max} <\infty
\end{equation}
a.e. in $\Omega_T$. In addition, we require the log-H\"{o}lder continuity of the function $p$ with respect to the spatial variable, but \emph{we do not postulate any assumption concerning its behavior with respect to the time variable}, see the next section for the precise formulation. Note that the problems with $p(t,x)$-growth may play an important role, when one analyzes many problems arising in physics. In particular, the continuum mechanics with electrorheological fluids serves as a prominent example. The properties of the material change with respect to the underlying electric field, which is typically a relatively smooth function in the spatial variable but may even be discontinuous in the time variable, see e.g. \cite{RAJAGOPAL1996401} and \cite{MR4578481} for relevant description of such models and see also reference therein. In general, the parabolic problems with a decently smooth data ($f\in L^2(\Omega_T)$ and $u_0\in L^2(\Omega)$), and a nonlinearity of the type~\eqref{sta1} were first considered in~\cite{MR4074614} under the framework of the operators with $(p,q)$-growth\footnote{In our setting, it would be $(p_{\min},p_{\max})$-growth.}. On the other hand, up to the full generality in terms of Orlicz--Sobolev spaces the problem was treated in  \cite{agnes14}, where the author assumed the general Orlicz--Sobolev setting, but required the log-H\"{o}lder-like continuity with respect to the $(t,x)$-variable. See also \cite{chlebicka2018well} for the generalization in the non-reflexive spaces, or \cite{MR3023393} for the more general operators. Up to date, the best result for the most general behavior of the power law exponent is due to~\cite{bulicek2021parabolic}, where the existence and the uniqueness of the weak solution is shown for the power-law variable $p$ fulfilling the same condition as in this paper (possibly discontinuous in time) and for good data $f$ and~$u_0$.

Second, we deal with very weak assumption on the data, which forces us to go beyond the theory of classical weak solutions, and we need to introduce different concepts. The starting point of the analysis is undoubtedly the series of works  \cite{MR1453181,MR1025884,MR1183665} dealing with the constant $p$-growth and the $L^1$ or measure data. Unfortunately, the concept of a distributional solution is too weak to provide the uniqueness and therefore a concept of the entropy solution was introduced for the integrable data in~\cite{MR1436364}.  Finally, the case when in the equation there appears a non-integrable, convective term was treated in~\cite{MR2729044}.  Note, that for problems with the (constant) $p$-growth, there is available an optimal regularity theory, and we refer the reader to the works \cite{MR2862024,MR4162964,MR2852212}.

However, there are only few results dealing with a $p(t,x)$-setting and integrable data. Let us mention the work \cite{gwiazda2015renormalized}, where such case is treated and the concept of a renormalized solution is introduced. Notice, that the renormalization technique goes hand in hand with the notion of an entropy solution, and therefore is also a possible candidate for a "proper" definition of a solution. Nevertheless, and most importantly, we are not aware of any results dealing with the integrable data and $p(t,x)$-growth with the growth parameter $p$ being possibly discontinuous with respect to the time variable. We would like to emphasize that such setting is not only a borderline case of some academic problem, but it is indeed a possible natural setting in cases, when one deals with the rheology depending discontinuously on the time variable (a prototype of this, is the electrorheological fluid).

The main result of the presented paper is \emph{the existence theory for weak solutions and the existence and uniqueness theory for entropy solutions} for the same range of parameters as it is available for the constant $p$-growth, provided that the $p(t,x)$ parameter is log-H\"{o}lder continuous with respect to the spatial variable.

In the next section, we introduce the basic notation, function spaces and the rigorous statement of our result.  Section~\ref{S3} is devoted to the proof of the existence, while Section~\ref{S4} deals with the uniqueness of an entropy solution. Several technical tools used in the paper are recalled in the Appendix, mostly without the proof. Appendix~\ref{app:musielaki} is related mostly to the Musielak--Orlicz spaces and their properties, in particular to the variable exponent Lebesgue spaces. Appendix~\ref{Ap2} is devoted to the classical results in the theory of PDE's.

\section{Notations, assumptions and statement of main results}\label{S2}

%Appendix~\ref{app:musielaki}
%Appendix~\ref{Ap2}

We recall, that we consider a Lipschitz domain $\Omega\subset \R^d$ in the $d-$dimensional space, and we define $\Omega_T:=\Omega\times (0,T)$. In what follows,  $x\in\Omega$ always denotes the spatial variable, while $t\in (0, T)$ is reserved for the time variable. We use the standard notation for the Sobolev and the Lebesgue spaces. For vector manipulations, we write $a\cdot b$ for a scalar product whenever $a,b\in \mathbb{R}^d$. Moreover, for $p\in [1, +\infty]$ we denote by $p'$ its H\"{o}lder conjugate. Throughout the paper, the notation $L^{p(t, x)}(\Omega_T)$ and $L^{p(t, x)}(\Omega_T;\mathbb{R}^d)$ is notorious and means the variable exponent spaces and vector-valued variable exponent spaces, respectively, see Appendix~\ref{app:musielaki} for the full explanation. Here, we would like to remind that $p(t,x)$ is at least measurable and satisfies~\eqref{sta1}. Similarly, the notation $\mathcal{M}(\Omega)$ or $\mathcal{M}(\Omega_T)$ denotes the space of Radon measures defined on $\Omega$ or $\Omega_T$ respectively, which can be defined as a dual space to the space of continuous function. For any $f\in \mathcal{M}(U)$ and $g\in C(U)$ we denote
$$
\langle f, g\rangle_{\mathcal{M}(U)}:= \int_U g \diff f.
$$
As a last remark, the constant $C$ appearing in the estimates is a universal constant that may vary from line to line, but depends only on the data. If needed, we  write the dependence explicitly.

\subsection{Assumption on the nonlinearity \texorpdfstring{$A$}{A} and the definition of solution}

%\noindent\textbf{\underline{Assumptions on the parabolic operator.}}

We start with the assumption on $A$. In fact, here we just specify what we mean by \eqref{prototypeA}, i.e., that $A$ behaves like a $p(t,x)$-Laplacian.

\begin{Ass}[Assumptions on $A$]\label{intro:ass_on_A} We assume that $A:\Omega_T  \times \R^d \to \mathbb{R}^d$ satisfies:
\begin{enumerate}[label=(A\arabic*)]
\item \label{intro:ass_on_A:continuity}$A$ is a Carath\'eodory mapping, i.e., for a.e. $(t,x) \in \Omega_T$, map $\mathbb{R}^d \ni \xi \mapsto A(t,x,\xi)$ is continuous and for all $\xi \in \mathbb{R}^d$, map $\Omega_T \ni (t,x) \mapsto A(t,x,\xi)$ is measurable,
\item \label{intro:ass_on_A:coercgr}(coercivity and growth bound) there is a positive constant $c$ and a function $h \in L^{\infty}(\Omega_T)$, such that for all $\xi \in \R^d$ and a.e. $(t,x) \in \Omega_T$
$$
|\xi|^{p(t,x)} + |A(t,x,\xi)|^{p'(t, x)} \leq c \, A(t,x,\xi)\cdot \xi + h(t,x),
$$
\item \label{intro:assumA_mono}(monotonicity) for all $\eta,\xi \in \R^d$ and a.e. $(t,x) \in \Omega_T$
$$
(A(t,x,\xi) - A(t,x,\eta)) \cdot (\xi - \eta) \geq 0,
$$
\item  \label{intro:assumA_vanish} for a.e. $(t,x) \in \Omega_T$ we have $A(t,x,0) = 0$.
\end{enumerate}
\end{Ass}

Having specified the behavior of $A$, we can now define the notion of a weak (or a very weak) solution.
\begin{Def}[Weak solutions]\label{def:weaksolution}
Let $f\in \mathcal{M}(\Omega_T)$,  $u_0\in \mathcal{M}(\Omega)$ and let $A$ satisfy Assumption~\ref{intro:ass_on_A}. We  say that $u: \Omega_T \rightarrow \R$ is a weak solution to the problem \eqref{eq:mainequation}--\eqref{eq:boundaryconditions} if $u\in L^\infty((0, T); L^1(\Omega))\cap L^1((0, T); W^{1, 1}_0(\Omega))$, $A(t, x, \nabla_x u)\in L^1(\Omega_T; \R^d)$ and
\begin{align}\label{wfweak}
    \int_{\Omega_T}-u\,\p_t\phi(t, x)\diff x\diff t + \int_{\Omega_T}A(t, x, \nabla_x u) \cdot \nabla_x \phi\diff x\diff t = \langle f,\phi\rangle_{\mathcal{M}(\Omega_T)} + \langle u_0, \phi(0, \cdot)\rangle_{\mathcal{M}(\Omega)},
\end{align}
for arbitrary $\phi\in C^\infty_c((-\infty, T)\times\Omega)$.
\end{Def}

The definition above is very standard, but it is not known how to prove the uniqueness of such solution. In addition, for its existence it is required a certain minimal value of $p$, which depends on the dimension $d$.
On the other hand, such limitation does not appear for the entropy solution, but it requires an introduction of the proper truncation function, which we do as follows. For any $k>0$ we define the classical cut-off function $T_k$ as
\begin{align}
    T_k(z) &= \left\{\begin{array}{ll} \label{def:trunc}
        z, \qquad&\text{ if }|z|\leq k \\
        \mathrm{sign}(z)k, \qquad&\text{ if }|z| > k
    \end{array}\right. ,\\
    \intertext{and we also define its $\varepsilon$-mollification as }
    T_{k,\varepsilon}(z)&= \left\{\begin{array}{ll}\label{def:trunc_smooth}
        z, \qquad&\text{ if }|z|\leq k, \\
        \mathrm{sign}(z)(k+\varepsilon/2), \qquad&\text{ if }|z| \geq k + \varepsilon,
    \end{array}\right. ,
\end{align}
whereas $T_{k,\varepsilon}$ is defined on $(k, k+\varepsilon)$ in such a way that $T_{k,\varepsilon}\in C^2(\R)$, $0 \leq T'_{k, \varepsilon}\leq 1$ and $T_{k,\varepsilon}$ is concave on $\R_+$, convex on $\R_-$ with a second derivative satisfying $|T''_{k,\varepsilon}|\leq C\varepsilon^{-1}$. For further reference we define as well the primitive function of $T_k$
\begin{align}\label{def:primitiveoftruncation}
G_k(s) = \int_0^s T_k(z)\diff z,
\end{align}
and
\begin{align}\label{chik}
%\begin{aligned}
\chi_k(z) = \left\{\begin{array}{ll}
        1, \qquad&\text{ if }|z|\leq k, \\
        0, \qquad&\text{ if }|z| > k.
    \end{array}\right. .
%\end{aligned}
\end{align}
Meaning, that $T'_k(z) = \chi_k(z)$ whenever  $z\neq \pm k$.

%\textbf{\underline{Definitions of solutions.}}

First, with the help of truncation function, we can weaken the definition of weak solution in the following way.
\begin{Def}[Weak solutions II]\label{wierd}
Let $f\in \mathcal{M}(\Omega_T)$,  $u_0\in \mathcal{M}(\Omega)$ and let $A$ satisfy Assumption~\ref{intro:ass_on_A}. We  say that $u: \Omega_T \rightarrow \R$ is a weak solution to the problem \eqref{eq:mainequation}--\eqref{eq:boundaryconditions} if $u\in L^\infty((0, T); L^1(\Omega))$, $\overline{A}\in L^1(\Omega_T;\R^d)$, for all $k\in \mathbb{N}$ we have $T_k(u)\in L^1((0, T); W^{1, 1}_0(\Omega))$, and
\begin{align}\label{wfweakII}
    \int_{\Omega_T}-u\,\p_t\phi(t, x)\diff x\diff t + \int_{\Omega_T}\overline{A} \cdot \nabla_x \phi\diff x\diff t = \langle f,\phi\rangle_{\mathcal{M}(\Omega_T)} + \langle u_0, \phi(0, \cdot)\rangle_{\mathcal{M}(\Omega)},
\end{align}
for arbitrary $\phi\in C^\infty_c((-\infty, T)\times\Omega)$. In addition, we require that
$$
\overline{A}=A(t,x,\nabla_x T_k(u)) \quad \textrm{a.e. on } \{|u|\le k\}.
$$
\end{Def}
The above definition allows us to prove the existence of weak solution for large range of $p$'s. Nevertheless, the best range of $p$ is obtained for an entropy solution defined below.
\begin{Def}[Entropy solutions]\label{def:entropysolution}
Let $f\in L^1(\Omega_T)$,  $u_0\in L^1(\Omega)$ and let $A$~satisfy Assumption~\ref{intro:ass_on_A}. We say that $u: \Omega_T\rightarrow \R$ is an entropy solution to the problem \eqref{eq:mainequation}--\eqref{eq:boundaryconditions} if $u\in L^\infty((0, T);L^1(\Omega))$ and  $T_k(u)\in L^1((0,T); W^{1,1}_0(\Omega))$ for all $k\in \R_+$. In addition, we require that $\nabla_x T_k(u)\in L^{p(t, x)}(\Omega_T;\R^d)$ and $A(t, x, \nabla_x T_k(u))\in L^{p'(t, x)}(\Omega_T;\R^d)$, and
\begin{equation}\label{ineq:inequalityinentropydefinition}
\begin{split}
    &\int_{\Omega}G_k(u(t, x) - \phi(t, x)) - G_k(u_0(x) - \phi(0, x))\diff x + \int_0^t\int_{\Omega}T_k(u - \phi)\,\p_t\phi(t, x)\diff x\diff \tau\\
    &\phantom{=} + \int_0^t\int_{\Omega}A(t, x, \nabla_x u)\cdot\nabla_x T_k(u - \phi)\diff x\diff \tau \leq \int_0^t\int_{\Omega}f\,T_k(u - \phi)\diff x\diff \tau,
\end{split}
\end{equation}
for all $k\in \R_+$, all $\phi\in C^\infty_c((-\infty, T)\times\Omega)$, and almost all $t\in (0, T)$.
\end{Def}

%\noindent\textbf{\underline{Assumptions on the exponent.}}
\subsection{Main results}

We introduce two results - for weak and for entropy solution. For each of it, we need to add certain restriction on the behavior of the variable exponent~$p$. We start with the assumptions for a weak solution.
\begin{Ass}[Assumptions for weak solutions]\label{ass:exponent_cont_space}
We assume that a measurable function $p(t,x): \Omega_T \to [1,\infty)$ satisfies the following:
\begin{enumerate}[label=(W\arabic*)]    \item\label{ass:cont} (continuity in space) $p(t, x)$ is a log-H\"older continuous functions on $\Omega$ uniformly in time, i.e., there is a constant $C$ such that for  all $t \in [0,T]$  and all $x, y \in \Omega$ and fulfilling $0<|x-y|<1$, we have
$$
|p(t,x) - p(t,y)| \leq -\frac{C}{\log|x-y|},
$$
\item\label{ass:usual_bounds_exp} (bounds) there holds that $ \frac{2d+1}{d+1} < p_{\mathrm{min}}\leq p(t,x)  \leq p_{\mathrm{max}} < +\infty$ for a.e. $(t,x) \in \Omega_T$.
\end{enumerate}
\end{Ass}
With such assumption, we can formulate the existence result for a weak solution.
\begin{thm}\label{THM_W}
    Assume that $u_0\in \mathcal{M}(\Omega)$ and $f\in \mathcal{M}(\Omega_T)$. Moreover, let $A$ satisfy Assumption~\ref{intro:ass_on_A} and $p$ satisfy  Assumption~\ref{ass:exponent_cont_space}. Then, there exists a weak solution to the system \eqref{eq:mainequation}--\eqref{eq:boundaryconditions} in the sense of  Definition~\ref{def:weaksolution}.

Moreover, if the assumption \ref{ass:usual_bounds_exp} is not satisfied but the minimal value of $p$ satisfies at least $p_{\min}>\frac{2d}{d+1}$, then there exists a weak solution in sense of Definition~\ref{wierd}.
\end{thm}

For the entropy solutions, the lower bound of the function $p(t, x)$ may be lowered, therefore in this case we  work with the following assumption.
\begin{Ass}[Assumptions for entropy solutions]\label{ass:exponent_cont_space_entr}
We assume that a measurable function $p(t,x): \Omega_T \to [1,\infty)$ satisfies the following:
\begin{enumerate}[label=(E\arabic*)]    \item\label{ass:cont_entr} (continuity in space) $p(t, x)$ is a log-H\"older continuous functions on $\Omega$ uniformly in time, i.e., there is a constant $C$ such that for  all $t \in [0,T]$  and all $x, y \in \Omega$ and fulfilling $0<|x-y|<1$, we have
$$
|p(t,x) - p(t,y)| \leq -\frac{C}{\log|x-y|},
$$
\item\label{ass:usual_bounds_exp_entr} (bounds) it holds that $ 1 < p_{\mathrm{min}}\leq p(t,x)  \leq p_{\mathrm{max}} < +\infty$ for a.e. $(t,x) \in \Omega_T$.
\end{enumerate}
\end{Ass}

%\noindent\textbf{\underline{Main result.}} Finally, we can state our main results.

The result for the entropy solution is the following.
\begin{thm}\label{THM_E}
    Assume that  $u_0\in L^1(\Omega)$ and $f\in L^1(\Omega_T)$. Moreover, let $A$ satisfy Assumption~\ref{intro:ass_on_A} and $p$ satisfy Assumption~\ref{ass:exponent_cont_space_entr}. Then, there exists a unique entropy solution to the system \eqref{eq:mainequation}--\eqref{eq:boundaryconditions} in the sense of Definition~\ref{def:entropysolution}.
\end{thm}

To finish this part, we want to emphasize once again, that for the constant exponent~$p$ we re-proved the classical and up-to-date optimal results. The main novelty of our work is that we do allow the exponent~$p$ to depend on the spatial and the time variable. However, and contrary to all the previous papers, we do not consider any smoothness of the exponent~$p$ with respect to the time variable, and we require only the classical (and maybe unavoidable) assumption on the log-H\"{o}lder continuity with respect to the spatial variable.

\section{Existence of a weak and an entropy solution}
\label{S3}

In this section, we provide the proof of the existence of a weak solution stated in the Theorem~\ref{THM_W} and an entropy solution stated in the Theorem~\ref{THM_E}. Since the vast majority of the considerations are the same in both cases, we do not distinguish between them. Instead, let us comment here on what differs. The bigger lower bound on the exponent $p(t, x)$ from \ref{ass:usual_bounds_exp}, when compared to \ref{ass:usual_bounds_exp_entr}, is needed to prove Lemma~\ref{lem:normboundsgradandoperator}, which gives weak compactness in the Lebesgue space $L^a$ for some $a > 1$ (see Remark \ref{rem:boundednessofnalbaxuandoperator}). This allows us to pass to the limit in the weak formulation for the approximation. As entropy solutions are defined with a truncation operator~\eqref{def:trunc}, we can establish the bounds for them much easier in the form of the Lemma \ref{lem:boundsontruncations}.

Before starting the proof, we recall the standard definition of the mollification, and we emphasize how we distinguish between mollification with respect to the spatial and the time variable.
%\begin{Def}[Mollification with respect to the spatial variable]\label{res:mol_in_sp}
Assume that $\eta:\R^d \to \R$ is a standard regularizing kernel, i.e., $\eta$ is a smooth, non-negative radially symmetric function, compactly supported in a ball of radius one and fulfills $\int_{\mathbb{R}^d} \eta(x) \diff x = 1$. Then, we set $\eta_{\kappa}(x) = \frac{1}{\kappa^d} \eta\left(\frac{x}{\kappa}\right)$ and for arbitrary $u: \R^d \times [0,T] \to \mathbb{R}$, we define
\begin{equation}\label{res:mol_in_sp}
u^{\kappa}(t,x) = \int_{\R^d} \eta_{\kappa}(x-y) u(t,y) \diff y.
\end{equation}
%\end{Def}
Similarly,
%\begin{Def}[Mollification with respect to time]\label{res:mol_in_ti}
assume that $\zeta:\R \to \R$ is a standard regularizing kernel, i.e., $\zeta$ is a smooth, non-negative function, compactly supported in a ball of radius one and fulfills $\int_{\mathbb{R}} \zeta(t) \diff t = 1$. Then, we set $\zeta_{\alpha}(t) = \frac{1}{\alpha} \zeta\left(\frac{t}{\alpha}\right)$ and for arbitrary $u: \R \times \Omega \to \R$, we define $\mathcal{R}^{\alpha}u: \R \times \Omega \to \R$ as
\begin{equation}\label{res:mol_in_ti}
\mathcal{R}^{\alpha}u(t,x) = \int_{\R} \zeta_{\alpha}(t-s)\, u(s,x) \diff s.
\end{equation}
%\end{Def}

%\noindent We may move to the formulations of the main theorems. But before that, we need the assumptions on the exponent $p(t, x)$, the parabolic operator $A$, and the definitions of solutions.\\
%\\

To finish this introductory part, we formulate and prove the technical result for the variable exponent needed in what follows.
\begin{lem}\label{lem:local_bounds}
    Suppose that $p$ satisfies Assumption~\ref{ass:exponent_cont_space} or~\ref{ass:exponent_cont_space_entr}. Then, for any $\varepsilon > 0$, there exists a radius $r_\varepsilon > 0$ and an open, finite covering $\{\mathcal{B}^{i}_{r_\varepsilon}\}_{i=1}^N$ of $\Omega$ with balls of radii $r_\varepsilon$, such that if we define
    $$
    q_i(t) := \inf_{\mathcal{B}^i_r}p(t, x), \quad r_i(t) := \sup_{\mathcal{B}^i_r}p(t, x),
    $$
    then
    \begin{align}\label{ineq:local_bounds}
        p_{\mathrm{min}}\leq q_i(t) \leq p(t, x) \leq r_i(t) < q_i(t) + \varepsilon, \text{ on }(0, T) \times (\mathcal{B}^i_{r_\varepsilon} \cap \Omega).
    \end{align}
\end{lem}
\begin{proof}
    We can cover $\Omega$ with a finite covering, and by \ref{ass:cont} or \ref{ass:cont_entr} respectively,  we may find a radius~$r_\varepsilon$, such that
    $$
    \sup_{\mathcal{B}^i_r}p(t, x) - \inf_{\mathcal{B}^i_r}p(t, x) \leq \varepsilon.
    $$
\end{proof}

\subsection{Approximate problem}

To begin the proof, we consider the approximate problem of the form
\begin{align}\label{eq:approxequation}
    \p_t u^n(t,x) - \DIV_x A(t, x, \nabla_x u^n) = f^n(t,x)
\end{align}
completed with the homogeneous Dirichlet boundary conditions and the initial conditions $u^n(0, x) = u_0^n(x)$. Here, $f^n\in L^\infty(\Omega_T)$, $u_0^n\in L^\infty(\Omega)$ and the sequences $\{f^n\}_{n\in\N}$, $\{u_0^n\}_{n\in \N}$ are taken in such a way that
\begin{align}
    f^n&\overset{\ast}{\rightharpoonup} f \quad\,\,\text{  weakly$^{\ast}$ in }\mathcal{M}(\Omega_T),\label{app:strongconvergenceoffM}\\
    u_0^n&\overset{\ast}{\rightharpoonup} u_0 \quad\text{ weakly$^{\ast}$ in }\mathcal{M}(\Omega),\label{app:strongconvergenceofboundarydataM}
\end{align}
so that we also have
\begin{equation}\label{L1E}
\sup_{n\in \mathbb{N}} \left(\int_{\Omega} |u_0^n|\diff x + \int_{\Omega_T} |f^n|\diff x \diff t\right) \le C<\infty.
\end{equation}
Moreover, in case of an entropy solution, we strengthen the above convergence results to
\begin{align}
    f^n&\longrightarrow f \quad\,\,\text{  strongly in }L^1(\Omega_T),\label{app:strongconvergenceoff}\\
    u_0^n&\longrightarrow u_0 \quad\text{ strongly in }L^1(\Omega)\label{app:strongconvergenceofboundarydata}.
\end{align}
For this approximation scheme, we can recall the following existence theorem for the approximative problem~\eqref{eq:approxequation}.
\begin{prop}\label{app:propositionofexistence}\textbf{\textup{(Theorem 1.23, Lemma 4.2, \cite{bulicek2021parabolic})}}
Let the operator~$A$ satisfy Assumptions~\ref{intro:ass_on_A} and let the exponent $p$ satisfy Assumption~\ref{ass:exponent_cont_space} or~\ref{ass:exponent_cont_space_entr}. Then there exists $u^n\in L^1((0, T); W^{1, 1}_0(\Omega))\cap L^\infty((0, T); L^2(\Omega))$ and \mbox{$\nabla_x u^n \in L^{p(t, x)}(\Omega_T; \R^d)$}, $A(t, x, \nabla_x u^n)\in L^{p'(t,x)}(\Omega_T; \R^d)$, such that
\begin{equation}\label{app:weakformulation}
\begin{split}
&-\int_{\Omega_T} u^n(t,x) \partial_t \varphi(t,x)  \diff x \diff t
- \int_{\Omega} u_0^n(x)\varphi(0,x) \diff x + \\
&\qquad \qquad \qquad \qquad \qquad+
\int_{\Omega_T} A(t,x,\nabla_x u^n) \cdot \nabla_x \varphi(t,x)  \diff x \diff t =
\int_{\Omega_T} f^n(t,x) \varphi(t,x)  \diff x \diff t
\end{split}
\end{equation}
for any $\varphi\in C^\infty_0([0, T)\times\Omega)$. Moreover, the following energy equality holds for almost all $t\in (0,T)$ (see \eqref{def:primitiveoftruncation} and \eqref{def:trunc} for the definition of $G_k$ and $T_k$ respectively)
\begin{equation}\label{app:global_energy_equality}
\begin{split}
&\int_{\Omega} \left[G_k(u^n(t,x)) - G_k(u_0^n(x)) \right] \diff x=\\
&  \qquad =
- \int_0^t \int_{\Omega} A(s,x, \nabla_x u^n) \cdot \nabla_x\left[T_k(u^n(s,x)) \right] \diff x \diff s + \int_0^t \int_{\Omega} f^n(s,x) \,T_k(u^n(s,x)) \, \diff x \diff s.
\end{split}
\end{equation}
\end{prop}
Now, our goal is to  let $n\to \infty$ in this approximate problem and to show that the limit is in fact a weak or an entropy solution respectively.

%\noindent\textbf{\underline{\`{A} priori bounds.}}
\subsection{Uniform bounds independent of \texorpdfstring{$n$}{n}}
We start by establishing some bounds, which are uniform with respect to $n$ and which depend only on the norms $\|f^n\|_1$ and $\|u^n_0\|_1$. These are, however, quantities controlled uniformly due to \eqref{L1E}. %The lemma below is enough for the proof of the existence of the entropy solutions.
\begin{lem}\label{lem:boundsontruncations}
Consider the situation as in Proposition~\ref{app:propositionofexistence} (and recall the definition of $\chi_k$ in  \eqref{chik}). Then, for every $k\in \mathbb{N}$ the following estimate  holds
\begin{align}\label{app:energytruncl1}
\begin{aligned}
    \sup_{t\in(0, T)}\|u^n(t)\|_1 + \int_{\Omega_T}|\nabla_x T_k(u^n)|^{p(t,x)} &+ |A(t, x, \nabla_x u^n)\chi_k|^{p'(t,x)}\diff x\diff t \\
    &\leq C(c, k, \|f^n\|_1, \|u_0^n\|_1, \|h\|_{\infty}).
\end{aligned}
\end{align}
In addition, for any $\lambda > 1$, we have
\begin{align}\label{app:fullenergyl1}
    \int_0^t\int_{\Omega}\frac{A(s, x, \nabla_x u^n) \cdot \nabla_x u^n}{(1 + |u^n|)^\lambda}\diff x\diff s\leq C\frac{\lambda}{\lambda - 1}\left(\|f^n\|_1 + \|u^n_0\|_1+1\right).
\end{align}
Furthermore, the estimate~\eqref{app:energytruncl1}, and the standard diagonal procedure imply, that there exists a subsequence $u^n$ (which we do not relabel), such that for all $k\in \mathbb{N}$, there are $B_k$ and $v_k$ fulfilling
\begin{align}
    A(t, x, \nabla_x T_k(u^n)) &\rightharpoonup B_k &&\text{ weakly in }L^{p'(t, x)}(\Omega_T; \R^d),\label{weaklimitoperatortrunc}\\
     \nabla_x T_k(u^n) &\rightharpoonup v_k &&\text{ weakly in }L^{p(t, x)}(\Omega_T; \R^d).\label{weaklimittrunaction}
\end{align}
\end{lem}
\begin{proof}
We start with the proof of~\eqref{app:energytruncl1}. First, we rewrite \eqref{app:global_energy_equality} as
\begin{multline}\label{app:energyrewritten}
    \|G_k(u^n(t))\|_1 + \int_0^t\int_{\Omega}A(s, x, \nabla_x u^n) \cdot \nabla_x u^n\chi_k \diff x\diff s \leq \int_0^t \int_{\Omega}|f^n||u^n|\chi_k\diff x\diff s\\
    + k\int_0^t\int_{\Omega}|f^n|(1 - \chi_k)\diff x\diff s + \|G_k(u_0^n)\|_1,
\end{multline}
where $\chi_k$ is just the abbreviation of the function $\chi_k(u(t,x))$, see also~\eqref{chik}. Then, using~\ref{intro:ass_on_A:coercgr} and a trivial estimate
\begin{align*}
    |u| - 1\leq G_1(u) \leq G_k(u) \leq C(k)|u|,
\end{align*}
we get the inequality~\eqref{app:energytruncl1}.

To obtain~\eqref{app:fullenergyl1}, we fix $\lambda>1$ and multiply~\eqref{app:energyrewritten} by $(1+k)^{-1-\lambda}$, and consider $t:=T$ in ~\eqref{app:energyrewritten}. Then, we integrate with respect to $k\in (0,\infty)$. Since the first term on the left-hand side is non-negative, we are led to the following
\begin{equation}
\label{app:estimates1}
\begin{split}
    %\int_0^\infty \int_{\Omega}\frac{G_k(u^n(t))}{(1+k)^{\lambda + 1}}\diff x \diff k +
    &\int_0^\infty\int_0^T\int_{\Omega}\frac{A(s, x, \nabla_x u^n)\cdot\nabla_x u^n\chi_k}{(1+k)^{1+\lambda}}\diff x\diff t\diff k
    \leq \int_0^\infty \int_{\Omega}\frac{G_k(u_0^n)}{(1+k)^{1+\lambda}}\diff x\diff k\\ &\qquad +\int_0^\infty\int_0^T\int_{\Omega}\frac{|f^n||u^n|\chi_k}{(1+k)^{1+\lambda}}\diff x\diff t\diff k  + \int_0^\infty\int_0^T\int_{\Omega}\frac{k|f^n|(1 - \chi_k)}{(1+k)^{\lambda+1}}\diff x\diff t\diff k.
\end{split}
\end{equation}
Let us evaluate the terms on the right-hand side. For arbitrary $w : \mathcal{O}\longrightarrow \R$, we employ the Fubini theorem to get
\begin{align*}
    \int_0^\infty\int_{\mathcal{O}}\frac{w(z)\chi_k(u(z))}{(1+k)^{\lambda+1}}\diff z\diff k &= \int_{\mathcal{O}}\int_{|u(z)|}^\infty\frac{w(z)}{(1+k)^{\lambda+1}} \diff k\diff z = \frac{1}{\lambda}\int_{\mathcal{O}}\frac{w(z)}{(1+|u(z)|)^{\lambda}}\diff z ,\\
%\end{align*}
\intertext{and}
%\begin{align*}
    \int_0^\infty\int_{\mathcal{O}}\frac{kw(z)(1-\chi_k(u(z)))}{(1+k)^{\lambda + 1}} &= \int_{\mathcal{O}}\int_0^{|u(z)|}\frac{kw(z)}{(1+k)^{\lambda + 1}}\diff k\diff z\\ &= \frac{1}{(1-\lambda)\lambda}\int_{\mathcal{O}}\frac{w(z) + \lambda w(z)|u(z)|}{(1+|u(z)|)^\lambda} - w(z)\diff z.
\end{align*}
Furthermore, since $G_k' = T_k$ we can deduce
\begin{align*}
    \int_0^\infty \int_{\Omega}\frac{G_k(u^n_0(x))}{(1+k)^{\lambda + 1}}\diff x \diff k &= \int_{\Omega}\int_0^{|u_0^n(x)|}\int_0^\infty\frac{T_k(s)}{(1+k)^{\lambda + 1}}\diff k\diff s\diff x\\
    &= \int_{\Omega}\int_0^{|u_0^n(x)|}\frac{1}{(1-\lambda)\lambda}((1+s)^{1 - \lambda} - 1)\diff s\diff x\\
    &= \int_{\Omega}\frac{(1+|u_0^n(x)|)^{2 - \lambda} - 1}{\lambda} - \frac{|u_0^n(x)|}{(1 - \lambda)\lambda}\diff x.
\end{align*}
Inserting these three identities into~\eqref{app:estimates1} we get
\begin{align*}
    \int_0^T\int_{\Omega}\frac{A(s, x, \nabla_x u^n)\cdot \nabla_x u^n}{(1 + |u^n|)^\lambda}\diff x\diff t &\leq\int_0^T\int_{\Omega}\frac{|f^n||u^n|}{(1+|u^n)^\lambda}\diff x\diff t + \lambda\int_0^\infty \int_{\Omega}\frac{G_k(u^n_0(x))}{(1+k)^{\lambda + 1}}\diff x \diff k\\
    &\quad + \int_0^T\int_{\Omega}\frac{|f^n|}{1 - \lambda}\left(\frac{1 + \lambda|u^n|}{(1 + |u^n|)^\lambda} - 1\right)\diff x\diff t \\
    &\hspace{-30pt}= \frac{1}{1 - \lambda}\int_0^T\int_{\Omega}\frac{|f^n|}{(1+|u^n|)^\lambda}\left(1 + |u^n|\right)\diff x\diff t\\
    &\quad \hspace{-30pt} +\frac{1}{\lambda - 1}\int_0^T\int_{\Omega}|f^n|\diff x\diff t+ \int_{\Omega}(1 + |u_0^n(x)|)^{2 - \lambda} - 1 + \frac{|u_0^n(x)|}{\lambda - 1}\diff x.
\end{align*}
As $\lambda > 1$, the inequality above proves the estimate \eqref{app:fullenergyl1}.
\end{proof}
\begin{rem}
The uniform bounds~\eqref{app:energytruncl1}--\eqref{app:fullenergyl1}, the assumption on data~\eqref{L1E}, and the structural assumption on the nonlinearity~\ref{intro:ass_on_A:coercgr} directly imply the following uniform estimate
\begin{align}\label{rem:remarkonptxbounds}
    \sup_{t\in (0, T)}\|u^n(t)\|_1 + \int_{\Omega_T}\frac{|\nabla_x u^n|^{p(t, x)} + |A(t, x, \nabla_x u^n)|^{p'(t, x)}}{(1 + |u^n|)^\lambda} \diff x\diff t\leq C(\lambda)
\end{align}
for any $\lambda > 1$. Note that the right-hand side explodes as $\lambda \to 1_{+}$.
\end{rem}

%\noindent Here, we may move to proving bounds needed for the existence of weak solutions. But first, we need to localize our argument in a specific way.

%Lemma~\ref{lem:local_bounds}

%\noindent With this in mind we may see the following.

The estimates above are sufficient to define the notion of an entropy solution. However, they cannot be used for the notion of a weak solution. Therefore, we use~\eqref{rem:remarkonptxbounds} and a proper interpolation technique to deduce certain bounds on $\{u^n\}_{n\in\N}$ and $\{A(t,x,\nabla_x u^n)\}_{n\in\N}$ in Sobolev--Bochner spaces.

\begin{lem}\label{lem:normboundsgradandoperator}Consider the situation as in Proposition~\ref{app:propositionofexistence} and assume that $p_{\min}>\frac{2d}{d+1}$. Then, for any $\lambda\in (1,p_{\min})$, there holds
\begin{equation}\label{app:normboundsgradandoperator}
\begin{split}
    \int_{\Omega_T} |u^n|^{p(t,x)-\lambda +\frac{p(t,x)}{d}} + |\nabla_x u^n|^{p(t,x) - \frac{\lambda d}{d+1}} + |A(t,x,\nabla_x u^n)|^{\frac{p(t,x) - \frac{\lambda d}{(d+1)}}{p(t,x)-1}}\diff x \diff t &\leq C(\lambda,\Omega_T).
\end{split}
\end{equation}
\end{lem}
\begin{proof}
We recall Lemma~\ref{lem:local_bounds} and consider the balls constructed there. Then, we fix one ball $\mathcal{B}^i_{r_\varepsilon}$ and recall the definition of $q_i(t)$ - the minimal value of $p(t,x)$ in $\mathcal{B}^i_{r_\varepsilon}$, and $r_i(t)$ - the maximal value of $p(t,x)$ in $\mathcal{B}^i_{r_\varepsilon}$. Note that Lemma~~\ref{lem:local_bounds} gives us the maximal size of balls, but we can always assume a smaller ball with the corresponding notation of the minimal values of the power-law exponent.  %Let us define $E_i := \{t\in (0, T)\,:\,q_i(t) < d$\} and $D_i := (0, T) \setminus E_i$. Then, due to the Gagliardo--Nirenberg inequality, Poincar\'{e}'s inequality and \eqref{rem:remarkonptxbounds} we obtain
Then, we recall the classical interpolation inequality
\begin{equation}\label{intl}
\|v\|^s_{L^s(\mathcal{B}^i_{r_\varepsilon})}\le C\left(\|v\|^s_{L^{\frac{q_i(t)}{q_i(t)-\lambda}}(\mathcal{B}^i_{r_\varepsilon})}+\|v\|^{s-q_i(t)}_{L^{\frac{q_i(t)}{q_i(t)-\lambda}}
(\mathcal{B}^i_{r_\varepsilon})}\|\nabla_x v\|^{q_i(t)}_{L^{q_i(t)}(\mathcal{B}^i_{r_\varepsilon})} \right),
\end{equation}
where $\lambda>1$ is sufficiently close to $1$ and $s=s(t)$ is defined as
$$
s:=q_i(t) + \frac{q^2_i(t)}{d(q_i(t)-\lambda)}.
$$
At this point, we needed to consider $q_i(t)>\frac{2d}{d+1}$, which is automatically met, since we have assumed $p_{\min}>\frac{2d}{d+1}$.

Next, we apply the interpolation inequality above to the function
$$
v:=(1+|u|^n)^{\frac{q_i(t)-\lambda}{q_i(t)}}.
$$
Doing so, and integrating with respect to the time $t\in (0,T)$, we deduce
\begin{equation}
\label{numero}
\begin{aligned}
\int_0^T &\int_{\mathcal{B}^i_{r_\varepsilon}}|u^n|^{q_i(t)-\lambda +\frac{q_i}{d}}\diff x \diff t \le \int_0^T \int_{\mathcal{B}^i_{r_\varepsilon}} \left((1+|u^n|)^{\frac{q_i(t)-\lambda}{q_i(t)}}\right)^s \diff x \diff t  =\int_0^T \|v\|_{L^s(\mathcal{B}^i_{r_\varepsilon})}^s \diff s\\
&\le C\int_0^T\|v\|^s_{L^{\frac{q_i(t)}{q_i(t)-\lambda}}(\mathcal{B}^i_{r_\varepsilon})}+\|v\|^{s-q_i(t)}_{L^{\frac{q_i(t)}{q_i(t)-\lambda}}
(\mathcal{B}^i_{r_\varepsilon})}\|\nabla_x v\|^{q_i(t)}_{L^{q_i(t)}(\mathcal{B}^i_{r_\varepsilon})} \diff t\\
&=C\int_0^T\|(1+|u|^n)\|^{q_i(t)-\lambda + \frac{q_i(t)}{d}}_{L^{1}(\mathcal{B}^i_{r_\varepsilon})}\diff t\\
&\quad +C\int_0^T\|(1+|u|^n)\|^{\frac{q_i(t)}{d}}_{L^{1}(\mathcal{B}^i_{r_\varepsilon})}\int_{\mathcal{B}^i_{r_\varepsilon}}
\left(\frac{q_i(t)-\lambda}{q_i(t)}\right)^{q_i(t)}\frac{|\nabla_x u^n|^{q_i(t)}}{(1+|u|^n)^{\lambda}} \diff x\diff t\\
&\le C\left(1+ \int_0^T\int_{\mathcal{B}^i_{r_\varepsilon}}\frac{|\nabla_x u^n|^{p(t,x)}}{(1+|u|^n)^{\lambda}} \diff x\diff t\right)\le C(\lambda),
\end{aligned}
\end{equation}
where for the last two inequalities we have used the \`{a}~priori bound~\eqref{rem:remarkonptxbounds}. 

At this point we are now prepared to prove~\eqref{app:normboundsgradandoperator}. Hence, let $\lambda_0 \in (1,p_{\min})$ be arbitrary. Set 
$$
\varepsilon:= \frac{d(\lambda_0-1)}{2(d+1)}\qquad \textrm{and}\qquad \lambda:=\frac{\lambda_0 +1}{2}>1
$$ 
and consider the finite covering of $\Omega$ from Lemma~\ref{lem:local_bounds} with such $\varepsilon$. Note that with such choice, we have for arbitrary $\mathcal{B}^i_{r_\varepsilon}$ and arbitrary $x\in \mathcal{B}^i_{r_\varepsilon}$ that
\begin{equation}\label{numero2}
\begin{split}
p(t,x)-\lambda_0 +\frac{p(t,x)}{d}&=q(t)-\lambda +\frac{q(t)}{d} +\lambda-\lambda_0+\frac{(p(t,x)-q(t))(d+1)}{d}\\
&\le q(t)-\lambda +\frac{q(t)}{d} +\lambda-\lambda_0+\frac{\varepsilon(d+1)}{d} =q(t)-\lambda +\frac{q(t)}{d}.
\end{split}
\end{equation}
Thus, with the help of inequality \eqref{numero2} and the estimate \eqref{numero} and using also the fact that the covering $\{\mathcal{B}^i_{r_\varepsilon}\}_{i=1}^N$ is finite, we have
$$
\begin{aligned}
\int_{\Omega_T} (1+|u^n|)^{p(t,x)-\lambda_0 +\frac{p(t,x)}{d}}\diff x \diff t &\le \sum_{i=1}^N \int_0^T \int_{\mathcal{B}^i_{r_\varepsilon}} (1+|u^n|)^{p(t,x)-\lambda_0 +\frac{p(t,x)}{d}}\diff x \diff t\\
&\le \sum_{i=1}^N \int_0^T \int_{\mathcal{B}^i_{r_\varepsilon}} (1+|u^n|)^{q(t)-\lambda +\frac{q(t)}{d}}\diff x \diff t \le N C(\lambda).
\end{aligned}
$$
Consequently, we deduced that for arbitrary $\lambda\in (1,p_{\min})$ there holds
\begin{equation}\label{estptx}
\int_{\Omega_T}|u^n|^{p(t,x)-\lambda +\frac{p(t,x)}{d}}\diff x \diff t\le C(\lambda,\Omega_T).
\end{equation}

Next, we focus on the sequences $\{\nabla_x u^n\}_{n\in\N}$ and $\{A(t,x,\nabla_x u^n)\}_{n\in\N}$. For any $\lambda\in (1,p_{\min})$, we set
\begin{equation}\label{dfexpt}
    \begin{split}
        \zeta (t,x)&:= p(t,x) - \frac{\lambda d}{d+1},\\
        \xi(t,x)&:= \frac{1}{p(t,x)-1}\left(p(t,x) - \frac{\lambda d}{(d+1)}\right).
    \end{split}
\end{equation}
Then, with the use of the Young inequality, we deduce that
%Fix $\zeta < q_i(t)$ and $\xi < r_i(t)'$. With this, we may estimate using Young's inequality
\begin{equation}\label{ineq:firstforgradient}
\begin{split}
\int_{\Omega_T}|\nabla_x u^n|^{\zeta}\diff x\diff t &= \int_{\Omega_T}\left(\frac{|\nabla_x u^n|^{p(t,x)}}{(1+ |u^n|)^\lambda}\right)^{\frac{\zeta(t,x)}{p(t,x)}}(1 + |u^n|)^{\frac{\zeta(t,x)\lambda}{p(t,x)}}\diff x\diff t\\
&\leq \int_{\Omega_T}\frac{|\nabla_x u^n|^{p(t,x)}}{(1+ |u^n|)^\lambda}+(1 + |u^n|)^{\frac{\zeta(t,x)\lambda}{p(t,x)-\zeta(t,x)}}\diff x\diff t\\
&= \int_{\Omega_T}\frac{|\nabla_x u^n|^{p(t,x)}}{(1+ |u^n|)^\lambda}+(1 + |u^n|)^{p(t,x)-\lambda +\frac{p(t,x)}{d}}\diff x\diff t\le C(\lambda,\Omega_T),
\end{split}
\end{equation}
where we have used the definition of $\zeta$ in \eqref{dfexpt}, the uniform bounds~\eqref{rem:remarkonptxbounds} and \eqref{estptx}. Next, to simplify forthcoming formula, we define 
$$
A^n(t,x):=A(t,x,\nabla_x u^n)
$$
and in a very similar manner as above, we  deduce that
\begin{equation}\label{ineq:firstfortensor}
\begin{split}
\int_{\Omega_T}|A^n|^\xi\diff x\diff t &= \int_{\Omega_T}\left(\frac{|A^n|^{p'(t,x)}}{(1+ |u^n|)^\lambda}\right)^{\frac{\xi(t,x)}{p'(t,x)}}(1 + |u^n|)^{\frac{\xi(t,x)\lambda}{p'(t,x)}}\diff x\diff t\\
&\leq \int_{\Omega_T}\frac{|A^n|^{p'(t,x)}}{(1+ |u^n|)^\lambda}+ (1 + |u^n|)^{\frac{\xi(t,x)\lambda}{p'(t,x)}\xi(t,x)}\diff x\diff t\\
&=\int_{\Omega_T}\frac{|A^n|^{p'(t,x)}}{(1+ |u^n|)^\lambda}+(1 + |u^n|)^{p(t,x)-\lambda +\frac{p(t,x)}{d}}\diff x\diff t\le C(\lambda,\Omega_T),
\end{split}
\end{equation}
where we have used the definition of $\xi$ in \eqref{dfexpt} and again the uniform bounds~\eqref{rem:remarkonptxbounds} and~\eqref{estptx}. To summarize, the estimates \eqref{estptx}, \eqref{ineq:firstforgradient} and \eqref{ineq:firstfortensor} together with the definition \eqref{dfexpt} lead to~\eqref{app:normboundsgradandoperator}, and the proof is completed.
\end{proof}

\begin{rem}\label{rem:boundednessofnalbaxuandoperator}
The assumption $p_{\min}>\frac{2d}{d+1}$ and the estimate \eqref{app:normboundsgradandoperator} allow us to conclude that the sequences $\{ u^n\}_{n\in\N}$ and $\{A(t, x, \nabla_x u^n)\}_{n\in\N}$ are bounded in $L^a$ for some $a > 1$. In addition, if we assume that $p_{\min}>\frac{2d+1}{d+1}$, i.e., the condition \ref{ass:usual_bounds_exp}, then the same holds true for the sequence $\{\nabla_x u^n\}_{n\in\N}$.
\end{rem}

\subsection{Convergence properties of the approximation}
%\noindent\textbf{\underline{Convergence of approximation.}}

We start this key subsection by introducing several notations. First, we extend the solution to the negative times and recall the following result.
\begin{lem}\label{lem:extension_by_0_u}\textbf{\textup{(Lemma 2.12, \cite{bulicek2021parabolic})}}
Let $u^n$ be as in Proposition \ref{app:propositionofexistence} and let us extend $u^n$ to $\overline{u^n}$ as follows:
\begin{align*}
    \overline{u^n}(t,x) = \begin{cases}
    0 &\text{ when }t > T,\\
    u^n(t, x) &\text{ when }t\in (0, T],\\
    u_0^n(x) &\text{ when }t \leq 0.
    \end{cases}
\end{align*}
In addition, let $\overline{f^n}$ and $\overline{A^n}$ denote the extension by $0$ outside of the time interval $(0,T)$ for the quantities  $f^n$, $A^n$ respectively. Then
\begin{equation}\label{eq:weak_equality_extensions}
\begin{split}
    \int_{-T}^T&\int_{\Omega}-\overline{u^n}\,\p_t \phi + \overline{A^n}\cdot \nabla_x \phi \diff x \diff t = \int_{-T}^T\int_{\Omega}\overline{f^n} \,\phi\diff x \diff t
\end{split}
\end{equation}
holds for any $\phi\in C^\infty_c((-T, T) \times \Omega)$. Moreover,
\begin{equation}
\label{tdp}
\overline{u^n}^\kappa \in W^{1, 1}((-T, T)\times \Omega')
\end{equation}
for  $\Omega'$ given as
$$
\Omega':=\{x\in \Omega; \; B_{\kappa}(x)\subset \Omega\}.
$$
\end{lem}
%Here, we have introduced the shortened notation
%$$
%A^n := A(t, x, \nabla_x u^n).
%$$
Let us also define the following quantities for later use:
\begin{equation}\label{aux_q}
    \begin{split}
        \omega^{n, m} &:= T_{k}(u^n) - T_{k}(u^m),\\
        \omega^{n, m}_{\varepsilon} &:= T_{k, \varepsilon}(u^n) - T_{k, \varepsilon}(u^m),\\
        \overline{\omega^{n, m}_\varepsilon} &:= T_{k, \varepsilon}(\overline{u^n}) - T_{k, \varepsilon}(\overline{u^m}),\\
        \overline{\omega^{n, m}_{\varepsilon,\kappa}} &:= T_{k, \varepsilon}(\overline{u^n}^\kappa) - T_{k, \varepsilon}(\overline{u^m}^\kappa).
    \end{split}
\end{equation}
Here, recall the definitions \eqref{def:trunc}, \eqref{def:trunc_smooth}, and that the superscript $\kappa$ is connected to the mollification in space (see the beginning of the Section \ref{S3}). Furthermore, we introduce a classical approximation of  the characteristic function of the time interval $(\eta, \beta)\subset (-T, T)$
\begin{align}\label{eq:def_fun_gamma}
\gamma_{\eta, \beta}^{\tau}(t) := \left\{
\begin{aligned}
     &0 &&\text{ for } t\leq \eta - \tau \text{ or } t\geq \beta + \tau, \\
     &1 &&\text{ for } \eta \leq t \leq \beta,\\
     &\text{affine} &&\text{ for }t\in [\eta - \tau, \eta]\cup[\beta, \varrho + \tau].
\end{aligned}
\right.
\end{align}
With this in mind, we start with the collection of the several convergence results. The first lemma states the almost everywhere convergence of the approximating sequence.
\begin{lem}\label{app:almosteverywhereconvergence}
Let $u^n$ be as in Proposition~\ref{app:propositionofexistence}. Then, there exists $u\in L^\infty((0, T); L^1(\Omega))$ such that, up to the subsequence which we do not relabel,
\begin{align}\label{pointA}
    u^n \longrightarrow u\quad\text{a. e. in }\Omega_T.
\end{align}
In particular, for any $k\in \mathbb{N}$ the weak limit $v_k$ defined in~\eqref{weaklimittrunaction} is equal to $\nabla_x T_k(u)$.
\end{lem}
\begin{proof}
In case that $p(t,x)$ is very close to one, we cannot apply the Aubin--Lions lemma directly to the sequence $\{u^n\}_{n\in\N}$ due to its low integrability and regularity. Therefore, we first focus on the sequence $\{T_{k,\varepsilon}(u^n) \}_{n\in \mathbb{N}}$, where $k\in \N$ and $\varepsilon>0$ are arbitrary, but fixed, and apply the Aubin--Lions lemma~\ref{aubin-lions} to such sequence.
%First, we focus on the cWe show that there exists $u\in L^\infty((0, T); L^1(\Omega))$, for which
%\begin{align}\label{app:almosteverywhereconvtrunc}
%T_{k,\varepsilon}(u^n) \longrightarrow T_{k,\varepsilon}(u) \text{ a.e. in }\Omega_T\quad\text{ for all }k\in \N \text{ and an arbitrary, fixed }\varepsilon > 0
%\end{align}
%is true. The point-wise convergence in~\eqref{pointA} is then a direct consequence of~\eqref{app:almosteverywhereconvtrunc}. To prove it, we want to use the Aubin--Lions lemma \ref{aubin-lions} applied to $\{T_{k,\varepsilon}(u^n) \}_{n\in \mathbb{N}}$. 
Indeed, thanks to \eqref{app:normboundsgradandoperator}, we see that the aforementioned sequence is bounded in $L^{p_{\min}}((0,T); W^{1,p_{\min}}_0(\Omega))$. Hence, we need a piece of information about its time derivative, which can be however deduced from the following identity
\begin{align}\label{eq:renormalizedequationforun}
\p_t T_{k,\varepsilon}(u^n) = \DIV_x \left(A(t, x, \nabla_x u^n)T'_{k,\varepsilon}(u^n)\right) + f^n\, T'_{k,\varepsilon}(u^n) - A(t, x,\nabla_x u^n)\cdot\nabla_x(T'_{k,\varepsilon}(u^n)).
\end{align}
The equation above is formally obtained by multiplying the equation~\eqref{eq:approxequation} by~$T'_{k,\varepsilon}(u^n)$. We skip the details here and refer the reader to the procedure explained in the proof of Lemma~\ref{lem:limitoftheproductinpowermu}, because the rigorous proof follows a similar line to the proofs below. % The idea is to test the equation \eqref{eq:weak_equality_extensions} with a test function
%$$
%\phi(t, x) = \mathcal{R}^\alpha\left(T'_{k,\varepsilon}(\overline{u^n}^\kappa)\varphi(t, x)\right)^\kappa
%$$
%and use \eqref{tdp} to deal with the term with time derivatives. To see how to handle the other problematic terms, one can look at \eqref{app:firstconvergencewithkappathirdterm} and \eqref{app:convergenceinkappathirdtermsecondone}.
Now, we may see that the right-hand side of \eqref{eq:renormalizedequationforun} is bounded in $L^1((0, T); (W_0^{2,d+1}(\Omega))^*)$. Indeed, it follows  from~\eqref{app:energytruncl1} that $\{A(t, x, \nabla_x u^n)T'_{k,\varepsilon}(u^n))\}_{n\in\N}$ is bounded in $L^{p_{\mathrm{max}}'}(\Omega_T)$, and from the Sobolev embeddings the spatial derivatives of the functions from $L^\infty((0, T); W_0^{2, d+1}(\Omega))$ are in $L^\infty((0, T); L^{\infty}(\Omega))$. Furthermore, $\{f^n\, T'_{k,\varepsilon}(u^n)\}_{n\in \N}$ and $\{A(t, x,\nabla_x u^n):\nabla_x(T'_{k,\varepsilon}(u^n))\}_{n\in\N}$ are both bounded in $L^1(\Omega_T)$. The former, from the definition of the sequence $\{f^n\}_{n\in\N}$, and the latter, from \eqref{app:fullenergyl1} and the fact that $\{T_{k,\varepsilon}''(u^n)(1 + |u^n|)^\lambda\}_{n\in\N}$ is bounded in $L^\infty(\Omega_T)$ for a fixed $\lambda > 1$. Hence, $\{\p_t T_{k,\varepsilon}(u^n)\}_{n\in \N}$ is bounded in $L^1((0, T); (W^{2,d+1}_0(\Omega))^*)$ and we may use the Aubin--Lions lemma~\ref{aubin-lions} to deduce (for a subsequence which we do not relabel)
\begin{align}\label{app:strongconvergenceinl1}
T_{k,\varepsilon}(u^n) \longrightarrow u_{k,\varepsilon} \text{ strongly in }L^1(\Omega_T)
\end{align}
for a fixed $k\in \N$, $\varepsilon>0$. Note here, that $u_{k,\varepsilon}$ are some functions (strong limits), but we do not know yet whether there is some underlying $u$ for which $T_{k,\varepsilon}(u)=u_{k,\varepsilon}$. On the other hand, we can easily let $\varepsilon\to 0_+$ and since $T_{k,\varepsilon} \to T_k$ uniformly on $\R$ (note also that $T_k$ is a bounded function) we have
\begin{align}\label{app:strongconvergenceinl1MB}
T_{k}(u^n) \longrightarrow u_{k} \text{ strongly in }L^1(\Omega_T)
\end{align}
for some functions $u_k \in L^{\infty}(\Omega_T)$. In addition, using the uniform bound~\eqref{rem:remarkonptxbounds} and the convergence result~\eqref{app:strongconvergenceinl1MB}, we observe that for any $t_1,t_2 \in [0,T]$ we have
\begin{equation}\label{donot}
\begin{split}
\int_{t_1}^{t_2}\|u_k(t)\|_{L^1(\Omega)}\diff t &\le \lim_{n\to \infty} \int_{t_1}^{t_2}\|T_{k}(u^n(t))\|_{L^1(\Omega)}\diff t  \le \lim_{n\to \infty} \int_{t_1}^{t_2}\|u^n(t)\|_{L^1(\Omega)}\diff t\\
&\le C(t_2-t_1).
\end{split}
\end{equation}
Moreover, using the definition of $T_k$, we have that for all $m>k$
$$
(T_k(u^n))_+ \le (T_m(u^n))_+, \qquad (T_k(u^n))_- \ge  (T_m(u^n))_-
$$
and consequently, we have the same for strong limits
$$
(u_k)_+ \le (u_m)_+, \qquad (u_k)_- \ge (u_m)_-
$$
whenever $k<m$. Hence, $\{(u_k)_+\}_{k\in \N}$ and $\{(u_k)_{-}\}_{k\in \N}$ are monotone sequences and using the uniform estimate \eqref{donot} and also the monotone convergence theorem, we may deduce that there exist $u\in L^1(\Omega_T)$ such that 
\begin{equation}\label{silenec}
u_k  \longrightarrow u \text{ strongly in }L^1(\Omega_T).
\end{equation}
Further, it follows from \eqref{donot} that 
\begin{equation}\label{donot2}
\begin{split}
\int_{t_1}^{t_2}\|u(t)\|_{L^1(\Omega)}\diff t \le C(t_2-t_1)
\end{split}
\end{equation}
and consequently, letting $t_2 \to t_1$, we see that for almost all $t\in (0,T)$, we have
$$
\|u(t)\|_{L^1(\Omega)}\le C,
$$
where $C$ is independent of $t$. Hence, we have proved that $u\in L^{\infty}((0,T); L^1(\Omega))$.

Finally, we show that 
$$
\lim_{n\to \infty} \int_{\Omega_T} \sqrt{|u^n-u|} \diff x \diff t=0
$$
and it implies that for a subsequence (that we again do not relabel) the convergence result~\eqref{pointA} holds true. Using the triangle inequality, the H\"{o}lder inequality, the uniform bound~\eqref{rem:remarkonptxbounds} and the convergence result~\eqref{app:strongconvergenceinl1MB}, we have
$$
\begin{aligned}
\lim_{n\to \infty} &\int_{\Omega_T} \sqrt{|u^n-u|} \diff x \diff t \le \lim_{n\to \infty} C\int_{\Omega_T} \sqrt{|u^n-T_k(u^n)|}+\sqrt{|T_k(u^n)-u_k|}+\sqrt{|u_k-u|} \diff x \diff t\\
&\le \lim_{n\to \infty} C\left( \|u^n\|_{L^1(\Omega_T)}^{\frac12} |\{|u^n|>k|\}^{\frac12} +\|u_k-u\|_{L^1(\Omega_T)}^{\frac12}\right) \le C\left( \frac{1}{k}+\|u_k-u\|_{L^1(\Omega_T)}\right)^{\frac12}.
\end{aligned}
$$
Thus, we let $k\to \infty$ and with the help of~\eqref{silenec}, we finish the proof.
%
%finish the proof it remains to sho
%
%recalling that $a_{+}$ denotes the positive part
%
%have  Hence, we may find a subsequence such that \eqref{app:almosteverywhereconvtrunc} is true for a fixed $k\in\N$. A classical diagonal argument gives us a subsequence such that \eqref{app:almosteverywhereconvtrunc} is true for all $k\in \mathbb{N}$. It remains to prove that $u\in L^\infty((0, T); L^1(\Omega))$. To see this, we may simply use the Fatou lemma in~\eqref{rem:remarkonptxbounds} with $\lambda=2$ and obtain
%$$
%\esssup_{t\in (0, T)}\|u\|_1 \leq \esssup_{t\in (0, T)}\|u^n\|_1 \leq C.
%$$
\end{proof}

After the preliminary convergence result, we now focus on the key result, which is, in fact, the~heart of the~proof.
\begin{lem}\label{lem:limitoftheproductinpowermu}
Let $\{u^n\}_{n\in \mathbb{N}}$ be as in Proposition~\ref{app:propositionofexistence}. Denote
$$
A^n_k:=A(t,x,\nabla_x T_k(u^n)).
$$
Then
\begin{align}\label{app:limitoftheproductinl1}
   \limsup_{n\to \infty} \limsup_{ m\to\infty}\int_{\Omega_T}\left|(A^n_k-A^m_k)\cdot(\nabla_x T_k(u^n) - \nabla_x T_k(u^m))\psi(x)\right|^\mu\diff x\diff t = 0
\end{align}
for any $\psi\in C_c^\infty(\Omega)$, any $0 < \mu < 1$ and arbitrary $k\in \mathbb{N}$.
\end{lem}
\begin{proof}
The proof heavily relies on the so-called $L^{\infty}$-truncation method, see~\cite{MR1453181}--\cite{MR1183665}, where it was applied to the parabolic problems with the constant exponent $p$. The second key ingredient is a proper approximation technique introduced in the context of a parabolic equation with variable, and possibly non-smooth, exponent $p(t,x)$ in~\cite{bulicek2021parabolic}.

We start with introducing the proper test functions in~\eqref{eq:weak_equality_extensions}. We set
\begin{align*}
\phi^n(t, x) &:= \left(T'_{k, \varepsilon}(\overline{u^n}^\kappa)\mathcal{R}^\alpha\left(T_{\delta}(\overline{\omega^{n,m}_{\varepsilon,\kappa}}^\kappa)\gamma_{-\eta,\beta}^\tau(t)\psi(x)\right)\right)^\kappa,\\
\phi^m(t, x) &:= \left(T'_{k, \varepsilon}(\overline{u^m}^\kappa)\mathcal{R}^\alpha\left(T_{\delta}(\overline{\omega^{n,m}_{\varepsilon,\kappa}}^\kappa)\gamma_{-\eta,\beta}^\tau(t)\psi(x)\right)\right)^\kappa,
\end{align*}
where $\psi\in C^{\infty}_c(\Omega)$, the parameters $\tau$ and $\alpha$ fulfill $-T<-\tau- \alpha -\eta<\beta+\tau+\alpha$ and the parameter $\kappa>0$ is so small that the support of $\psi\subset \Omega_{\kappa}$, where
$$
\Omega_{\kappa}:=\{x\in \Omega; \; B_{\kappa}(x)\subset \Omega\}.
$$
We also consider $k\in \mathbb{N}$, $\varepsilon\in (0,1)$ and $\delta\in (0,1)$ arbitrary. Here, recall the definitions \eqref{aux_q}, \eqref{def:trunc}, \eqref{def:trunc_smooth}, \eqref{eq:def_fun_gamma} and that the superscript $\kappa$ and the function $\mathcal{R}^\alpha$ are connected with the mollifications in space and time (see the beginning of the Section \ref{S3}). Setting now $\phi:=\phi^n$ in the equation \eqref{eq:weak_equality_extensions} for $n$ and setting $\phi:=\phi^m$ in the equation for \eqref{eq:weak_equality_extensions} for $m$ and subtracting both identities, we obtain
%
%by $\phi^n$ and subtract from it the analogous equation with $u^m$ replacing $u^n$ and $\phi^m$ replacing $\phi^n$. We obtain three terms.\\
\begin{equation}\label{MB1}
\begin{split}
    \underset{I_1}{\underbrace{\int_{-T}^T\int_{\Omega} \overline{A^n}\cdot \nabla_x \phi^n-\overline{A^m}\cdot \nabla_x \phi^m  \diff x \diff t}} &= \underset{I_2}{\underbrace{\int_{-T}^T\int_{\Omega}\overline{f^n} \,\phi^n-\overline{f^m} \,\phi^m\diff x \diff t}}\\
    &\quad -\underset{I_3}{\underbrace{\int_{-T}^T\int_{\Omega}-\overline{u^n}\,\p_t \phi^n+\overline{u^m}\,\p_t \phi^m \diff x \diff t}}.
\end{split}
\end{equation}
Next, we estimate terms on the right-hand side. We start with the last one, where we use integration by parts to get
$$
 -I_3= \int_{-T}^T\int_{\Omega}\p_t[T_{k,\varepsilon}(\overline{u^n}^\kappa) - T_{k,\varepsilon}(\overline{u^m}^\kappa)]\mathcal{R}^\alpha\left(T_{\delta}(\overline{\omega^{n,m}_{\varepsilon,\kappa}}^\kappa)\gamma_{\eta,\beta}^\tau(t)\psi(x)\right)\diff x\diff t.
$$
Making use of the Lebesgue dominated convergence theorem and the information \eqref{tdp},  we may let $\alpha,\,\tau\to 0^+$ and get
\begin{align*}
    \lim_{\alpha,\tau \to 0_+}-I_3=\int_{\Omega}\int_{-\eta}^\beta\p_t[T_{\delta}(\overline{\omega^{n,m}_{\varepsilon,\kappa}}^\kappa)]\psi(x)\diff x\diff t.
\end{align*}
Consequently, thanks to the definition of $T_{\delta}$, we see that
\begin{align}\label{boundsfirstterm}
    \lim_{\alpha,\tau \to 0_+}|I_3|\le \delta \|\psi\|_{L^{\infty}(\Omega)}.
\end{align}
Notice that the estimate is independent of $n,m,\kappa$.

The term $I_2$ can be estimated in an even simpler way as follows. First, we have
\begin{align*}
    I_2 = \int_{0}^T\int_{\Omega}f^n\,\phi^n - f^m\,\phi^m\diff x\diff t
\end{align*}
and by the H\"{o}lder inequality and the definition of $\phi^n$
\begin{align*}
    |I_2|\le 2\sup_{n\in \mathbb{N}} \|f^n\|_{L^1(\Omega_T)}\sup_{n\in \mathbb{N}} \|\phi^n\|_{L^{\infty}(\Omega_T)}\leq   2\sup_{n\in \mathbb{N}} \|f^n\|_{L^1(\Omega_T)}\sup_{n\in \mathbb{N}}\|T_{\delta}(\overline{\omega^{n,m}_{\varepsilon,\kappa}}^\kappa)\gamma_{\eta,\beta}^\tau(t)\psi(x)\|_{L^{\infty}(\Omega_T)}.
\end{align*}
Therefore, we obtain with the help of the uniform bound \eqref{L1E} and the trivial estimate for $T_{\delta}$, that
\begin{align}\label{boundssecondterm}
    |I_2|\le C\delta \|\psi\|_{L^{\infty}(\Omega)}.
\end{align}

Finally, we focus on the most important term $I_1$.
Similarly as before we may use the Lebesgue dominated convergence theorem and let $\alpha,\,\tau\to 0^+$, so defining $Y^{n,m,k}_{\beta,\kappa,\varepsilon,\delta}:=\lim_{\tau,\alpha \to 0_+}I_1$,  we have
\begin{align*}
Y^{n,m,k}_{\beta,\kappa,\varepsilon,\delta}&=\int_{0}^\beta\int_{\Omega}A^n\cdot\nabla_x\left[T'_{k,\varepsilon}((u^n)^\kappa)(T_{\delta}((\omega^{n, m}_{\varepsilon,\kappa})^\kappa)\psi)^\kappa\right]^\kappa \diff x\diff t\\
&\quad -\int_{0}^\beta \int_{\Omega} A^m\cdot\nabla_x\left[T'_{k,\varepsilon}((u^m)^\kappa)(T_{\delta}((\omega^{n, m}_{\varepsilon,\kappa})^\kappa)\psi)^\kappa\right]^\kappa\diff x\diff t
\end{align*}
and it follows from \eqref{boundsfirstterm} and \eqref{boundssecondterm} that
\begin{equation}
    \label{boundY}
  Y^{n,m,k}_{\beta,\kappa,\varepsilon,\delta}\le C\delta \|\psi\|_{L^{\infty}(\Omega)},
\end{equation}
where $C$ is independent of $n,m,\beta,\kappa,\varepsilon, \delta$. Our goal is to pass with $\kappa$ and $\varepsilon$ to $0$ and obtain the proper estimate on $Y^{n,m,k}_{\beta,\kappa,\varepsilon,\delta}$ from below. First, using the standard properties of mollification and the relation between $\kappa$ and the support of $\psi$, we get
\begin{align*}
Y^{n,m,k}_{\beta,\kappa,\varepsilon,\delta}&=\int_{0}^\beta\int_{\Omega}(A^n)^{\kappa}\cdot\nabla_x\left[T'_{k,\varepsilon}((u^n)^\kappa)(T_{\delta}((\omega^{n, m}_{\varepsilon, \kappa})^\kappa)\psi)^\kappa\right] \diff x\diff t\\
&\quad -\int_{0}^\beta \int_{\Omega} (A^m)^{\kappa}\cdot\nabla_x\left[T'_{k,\varepsilon}((u^m)^\kappa)(T_{\delta}((\omega^{n, m}_{\varepsilon, \kappa})^\kappa)\psi)^\kappa\right]\diff x\diff t\\
&=\int_{0}^\beta\int_{\Omega} \left(A^n\right)^\kappa\cdot\nabla_x(u^n)^\kappa\,T''_{k,\varepsilon}((u^n)^\kappa)(T_{\delta}((\omega^{n, m}_{\varepsilon, \kappa})^\kappa)\psi)^\kappa\diff x\diff t \\
&\quad -\int_0^{\beta}\int_{\Omega}
\left(A^m\right)^\kappa\cdot\nabla_x(u^m)^\kappa\,T''_{k,\varepsilon}((u^m)^\kappa)(T_{\delta}((\omega^{n, m}_{\varepsilon, \kappa})^\kappa)\psi)^\kappa\diff x\diff t\\
& + \int_0^\beta\int_{\Omega}((A^n)^\kappa T'_{k, \varepsilon}((u^n)^\kappa)- (A^m)^\kappa T'_{k, \varepsilon}((u^m)^\kappa))\cdot (\nabla_x (\omega^{n, m}_{\varepsilon, \kappa})^\kappa T'_{\delta}((\omega^{n, m}_{\varepsilon, \kappa})^\kappa)\psi(x))^\kappa\diff x\diff t\\
&+ \int_0^\beta\int_{\Omega}(A^n)^\kappa\,T'_{k,\varepsilon}((u^n)^\kappa)\cdot(T_{\delta}((\omega^{n, m}_{\varepsilon, \kappa})^\kappa)\nabla_x\psi)^\kappa \diff x\diff t\\
&\quad - \int_0^{\beta} \int_{\Omega}(A^m)^\kappa\,T'_{k,\varepsilon}((u^m)^\kappa)\cdot(T_{\delta}((\omega^{n, m}_{\varepsilon, \kappa})^\kappa)\nabla_x\psi)^\kappa\diff x\diff t.
\end{align*}
Next, we want to let $\kappa\to 0_+$. We  move forwards by proving two limits. For the first one, we want to show
\begin{equation}
\begin{split}\label{app:firstconvergencewithkappathirdterm}
\int_{0}^\beta\int_{\Omega} \left(A^n\right)^\kappa\cdot\nabla_x(u^n)^\kappa\,T''_{k,\varepsilon}((u^n)^\kappa)(T_{\delta}((\omega^{n, m}_{\varepsilon, \kappa})^\kappa)\psi)^\kappa\diff x\diff t\\
\longrightarrow \int_0^\beta\int_{\Omega}A^n\cdot\nabla_x u^n\,T''_{k,\varepsilon}(u^n)T_{\delta}(\omega^{n, m}_\varepsilon)\psi\diff x\diff t
\end{split}
\end{equation}
as $\kappa\to 0_+$. Due to our choice of $\kappa$ and $\psi$, we may write
\begin{align*}
&\left|\int_{0}^\beta\int_{\Omega} (A^n)^\kappa\cdot\nabla_x(u^n)^\kappa\,T''_{k,\varepsilon}((u^n)^\kappa)(T_{\delta}((\omega^{n, m}_{\varepsilon, \kappa})^\kappa)\psi)^\kappa -A^n\cdot\nabla_x u^n\,T''_{k,\varepsilon}(u^n)T_{\delta}(\omega^{n, m}_\varepsilon)\psi \diff x\diff t\right| \\
&\leq\|T''_{k,\varepsilon}((u^n)^\kappa)(T_{\delta}((\omega^{n, m}_{\varepsilon, \kappa})^\kappa)\psi)^\kappa\|_{L^{\infty}(\Omega_T)}\|\left(A^n\right)^\kappa\cdot\nabla_x(u^n)^\kappa\psi - A^n\cdot\nabla_x u^n\psi\|_{L^1(\Omega_T)}\\
&+ \int_0^\beta\int_{\Omega}\left|T''_{k,\varepsilon}((u^n)^\kappa)(T_{\delta}((\omega^{n, m}_{\varepsilon, \kappa})^\kappa)\psi)^\kappa - T''_{k,\varepsilon}(u^n)T_{\delta}(\omega^{n, m}_\varepsilon)\psi\right||A^n\cdot\nabla_x u^n|\diff x\diff t.
\end{align*}
Since from our assumptions
$$
\|T''_{k,\varepsilon}((u^n)^\kappa)(T_{\delta}((\omega^{n, m}_{\varepsilon, \kappa})^\kappa)\psi)^\kappa\|_{L^{\infty}(\Omega_T)}\leq C(\varepsilon, k, \delta, \psi),
$$
the first term may be treated by Proposition \ref{prop:convergence_of_convolution_in_musielak_orlicz} and Theorem \ref{thm:modular_l1} and the second one with the Lebesgue dominated convergence theorem. In the end, we obtain \eqref{app:firstconvergencewithkappathirdterm}. Similarly, we may prove
\begin{equation}\label{may}
\begin{split}
    \int_0^\beta\int_{\Omega}((A^n)^\kappa T'_{k, \varepsilon}((u^n)^\kappa)- (A^m)^\kappa T'_{k, \varepsilon}((u^m)^\kappa))\cdot (\nabla_x (\omega^{n, m}_{\varepsilon, \kappa})^\kappa T'_{\delta}((\omega^{n, m}_{\varepsilon, \kappa})^\kappa)\psi(x))^\kappa\diff x\diff t\\
    \to \int_0^\beta\int_{\Omega}
    A^n T'_{k, \varepsilon}(u^n)- A^m T'_{k, \varepsilon}(u^m)\cdot \nabla_x \omega^{n, m} T'_{\delta}(\omega^{n, m}_\varepsilon)\psi(x)\diff x\diff t
\end{split}
\end{equation}
and also
\begin{equation}
\begin{split}\label{app:convergenceinkappathirdtermsecondone}
    \int_0^\beta\int_{\Omega}(A^n)^\kappa\, T'_{k,\varepsilon}((u^n)^\kappa)\cdot(T_{\delta}((\omega^{n, m}_{\varepsilon, \kappa})^\kappa)\nabla_x\psi)^\kappa\diff x\diff t\\
    \rightarrow \int_0^\beta\int_{\Omega}A^n\, T'_{k,\varepsilon}(u^n)\cdot T_{\delta}(\omega^{n, m}_\varepsilon)\nabla_x\psi\diff x\diff t.
\end{split}
\end{equation}
Now, with the use of \eqref{boundsfirstterm}--\eqref{app:convergenceinkappathirdtermsecondone} we get
\begin{equation}\label{boundsbeforeconvergencewithepsilon}
\begin{split}
Y^{n,m,k}_{\beta,\varepsilon,\delta}&:=\int_0^\beta\int_{\Omega}(A^n T'_{k, \varepsilon}(u^n)- A^m T'_{k, \varepsilon}(u^m))\cdot \nabla_x \omega^{n, m}_\varepsilon\, T'_{\delta}(\omega^{n, m}_\varepsilon)\psi(x)\diff x\diff t\\
    &\leq \int_0^\beta\int_{\Omega}A^m\cdot\nabla_x u^m\,T''_{k,\varepsilon}(u^m)T_{\delta}(\omega^{n, m}_\varepsilon)\psi - A^n\cdot\nabla_x u^n\,T''_{k,\varepsilon}(u^n)T_{\delta}(\omega^{n, m}_\varepsilon)\psi\diff x\diff t\\
    &+ \int_0^\beta\int_{\Omega}A^m\,T'_{k,\varepsilon}(u^m)\cdot T_{\delta}(\omega^{n, m}_\varepsilon)\nabla_x\psi - A^n\,T'_{k,\varepsilon}(u^n)\cdot T_{\delta}(\omega^{n, m}_\varepsilon)\nabla_x\psi\diff x\diff t \\
    &\qquad + C\delta\|\psi\|_{L^{\infty}(\Omega)}.
\end{split}
\end{equation}
At this point, we want to let $\varepsilon\to 0_+$. It is rather standard in the term on the left-hand side and also in the second term on the right-hand side, so we have
\begin{equation}
\begin{split}
&Y^{n,m,k}_{\beta,\delta}:=\int_0^\beta\int_{\Omega}(A(t, x, \nabla_x T_{k}(u^n))- A(t, x, \nabla_x T_{k}(u^m))\cdot \nabla_x T_{\delta}(\omega^{n, m})\psi\diff x\diff t\\
&\leq \limsup_{\varepsilon \to 0_+}\int_0^\beta\int_{\Omega}A^m\cdot\nabla_x u^m\,T''_{k,\varepsilon}(u^m)T_{\delta}(\omega^{n, m}_\varepsilon)\psi - A^n\cdot\nabla_x u^n\,T''_{k,\varepsilon}(u^n)T_{\delta}(\omega^{n, m}_\varepsilon)\psi\diff x\diff t\\
&+ \int_0^\beta\int_{\Omega}T_{\delta}(\omega^{n, m})\left(A(t, x, \nabla_x T_{k}(u^m))- A(t, x, \nabla_x T_{k}(u^m))\right)\cdot \nabla_x\psi\diff x\diff t \\
&\qquad + C\delta\|\psi\|_{L^{\infty}(\Omega)}\\
&\leq \limsup_{\varepsilon \to 0_+}\int_0^\beta\int_{\Omega}A^m\cdot\nabla_x u^m\,T''_{k,\varepsilon}(u^m)T_{\delta}(\omega^{n, m}_\varepsilon)\psi - A^n\cdot\nabla_x u^n\,T''_{k,\varepsilon}(u^n)T_{\delta}(\omega^{n, m}_\varepsilon)\psi\diff x\diff t\\
&\qquad + C(k)\delta\|\psi\|_{W^{1,\infty}(\Omega)},
\end{split}\label{abc}
\end{equation}
where for the last inequality we have used \eqref{rem:remarkonptxbounds}.

To converge in the first term on the right-hand side of \eqref{abc}, we need to deal with the second derivative of $T_{k,\varepsilon}$, which is surely unbounded. To this end, we once again use the equation \eqref{eq:weak_equality_extensions} and test it with
$$
\phi := \left(T'_{k,\varepsilon}(\overline{u^n}_{+}^\kappa)\gamma_{-\eta,\beta}^\tau(t)\psi(x)\right)^\kappa,
$$
where non-negative $\psi$ is the same as above, $\kappa$ is sufficiently small and $\overline{u^n}^{\kappa}_{+}:=\max\{0,\overline{u^n}^{\kappa}\}$ denotes the positive part. Doing so, we have
\begin{equation*}%\label{druha}
\begin{split}
     -\int_{-T}^T\int_{\Omega}(\overline{A^n})^{\kappa} \cdot \nabla_x \phi \diff x \diff t = -\int_{-T}^T\int_{\Omega}\overline{f^n} \,\phi+\overline{u^n}\,\p_t \phi \diff x \diff t
\end{split}
\end{equation*}
and letting $\tau\to 0_+$ we easily obtain, after integration by parts with respect to the time variable, that
\begin{equation}\label{druha}
\begin{split}
     &-\int_{0}^\beta\int_{\Omega}(A^n)^{\kappa}\cdot \nabla_x(u^n)^\kappa_{+} T''_{k,\varepsilon}((u^n)^\kappa_{+})\psi(x) \diff x \diff t = -\int_{0}^\beta\int_{\Omega}(f^n)^{\kappa}\,T'_{k,\varepsilon}((u^n)^\kappa_{+})\psi(x)\diff x \diff t \\
     &\qquad  -\int_{0}^\beta\int_{\Omega}\p_t(u^n)^{\kappa}\,T'_{k,\varepsilon}((u^n)^\kappa_{+})\psi(x) \diff x \diff t +\int_{0}^\beta\int_{\Omega}(A^n)^{\kappa}\cdot \nabla_x \psi(x) T'_{k,\varepsilon}((u^n)^\kappa_{+}) \diff x \diff t.
\end{split}
\end{equation}

Next, we estimate all the terms on the right-hand side and let $\kappa\to 0_+$. We start with the time derivative.
\begin{align*}
\int_{0}^\beta\int_{\Omega}\p_t(u^n)^{\kappa}\,T'_{k,\varepsilon}((u^n)^\kappa_{+})\psi(x) \diff x \diff t  &= \int_{0}^\beta\int_{\Omega}\p_tT_{k,\varepsilon}((u^n)^\kappa_{+})\psi(x) \diff x \diff t\\
&= \int_{\Omega}T_{k,\varepsilon}((u^n)^\kappa_{+}(\beta))\psi(x) \diff x -\int_{\Omega}T_{k,\varepsilon}((u^n)^\kappa_{+}(0)\psi(x) \diff x.
\end{align*}
Using the definition of $T_{k,\varepsilon}$, and the fact that $\varepsilon\in (0,1)$, we deduce
\begin{align}\label{firstlimitforepsilonbound}
    \left|\limsup_{\kappa\to 0_+}\int_{0}^\beta\int_{\Omega}\p_t(u^n)^{\kappa}\,T'_{k,\varepsilon}((u^n)^\kappa_{+})\psi(x) \diff x \diff t \right|\leq Ck\|\psi\|_{L^{\infty}(\Omega)}.
\end{align}
Next, we may rewrite
\begin{align*}
    \int_{-T}^T\int_{\Omega}\overline{f^n}\,\phi\diff x\diff t = \int_0^\beta\int_{\Omega}f^n(T'_{k,\varepsilon}((u^n)^\kappa)\psi(x))^\kappa\diff x\diff t,
\end{align*}
and using the fact that $|T'_{k,\varepsilon}|\le 1$, the uniform bound \eqref{L1E}, and the H\"{o}lder inequality, we have
%$$
%\|(T'_{k,\varepsilon}%((u^n)^\kappa)\psi(x))^\kappa\|_\infty \leq C,
%$$
%we arrive at the conclusion
\begin{align}\label{secondlimitforepsilonbound}
    \left| \limsup_{\kappa\to 0_+}\int_{0}^\beta\int_{\Omega}(f^n)^{\kappa}\,T'_{k,\varepsilon}((u^n)^\kappa_{+})\psi(x)\diff x \diff t\right|\leq C\|\psi\|_{L^{\infty}(\Omega)}.
\end{align}
For the remaining term on the right-hand side of \eqref{druha}, we first use Theorem~\ref{thm:modular_l1} to let $\kappa\to 0_+$,  then the uniform bound~\eqref{rem:remarkonptxbounds}, as well as the fact that $T'_{k,\varepsilon}(s)=0$ whenever $|s|>k+1$, to deduce
\begin{equation}
\begin{aligned}\label{thirdlimitforepsilonbound}
  \lim_{\kappa\to 0_+}  &\left|\int_{0}^\beta\int_{\Omega}(A^n)^{\kappa}\cdot \nabla_x \psi(x) T'_{k,\varepsilon}((u^n)^\kappa_{+}) \diff x \diff t \right|\\
  &=  \left|\int_{0}^\beta\int_{\Omega}\frac{A^n}{1+|u^n|}\cdot \nabla_x \psi(x) (1+|u^n|) T'_{k,\varepsilon}((u^n)_{+}) \diff x \diff t \right|\\
  &\le Ck\|\psi\|_{W^{1,\infty}(\Omega)}\left(1+\int_{0}^\beta\int_{\Omega}\frac{|A^n|^{p(t,x)}}{(1+|u^n|)^{p_{\min}}} \diff x \diff t \right)\le Ck\|\psi\|_{W^{1,\infty}(\Omega)}.
\end{aligned}
\end{equation}
%
%
%\underline{Third term:}
%\begin{align*}
%    \int_{-T}^T\int_{\Omega}\overline{A^n}\cdot\nabla_x\varphi\diff x\diff t = \int_0^T\int_{\Omega}A^n\cdot((T'_{k,\varepsilon}((u^n)^\kappa)\,\nabla_x\psi + \nabla_x(u^n)^\kappa\,T''_{k,\varepsilon}((u^n)^\kappa)\psi(x))^\kappa\diff x\diff t.
%\end{align*}
Finally,  we may repeat the proof of \eqref{app:firstconvergencewithkappathirdterm} to pass to the limit with $\kappa \to 0_+$ in the term on the left hand side
\begin{align}
\begin{aligned}\label{cruciallimitforepsilonbound}
    &\lim_{\kappa\to 0_+}-\int_0^T\int_{\Omega}(A^n)^\kappa\cdot \nabla_x(u^n)^\kappa_+\,T''_{k,\varepsilon}((u^n)^\kappa_+)\psi(x)\diff x\diff t \\
    &\qquad = -\int_0^T\int_{\Omega}A^n\cdot\nabla_x u^n_+\,T''_{k,\varepsilon}(u^n_+)\psi(x)\diff x\diff t\\
    &\qquad =\int_0^T\int_{\Omega}|A(t,x,\nabla_x u^n)\cdot\nabla_x u^n_+\,T''_{k,\varepsilon}(u^n_+)\psi(x)|\diff x\diff t,
\end{aligned}
\end{align}
where for the last identity, we have used the fact that $T''_{k,\varepsilon}(s)\le 0$ for all $s\ge 0$ and the monotonicity of $A$, i.e., \eqref{intro:assumA_mono}--\eqref{intro:assumA_vanish}.
%
%At the same time, with the use of H\"{o}lder's and Young convolution inequalities and \eqref{app:energytruncl1}
%\begin{align}\label{thirdlimitforepsilonbound}
%    \left|\int_0^T\int_{\Omega}A^n\cdot(T'_{k,\varepsilon}((u^n)^\kappa))\,\nabla_x\psi)^\kappa\diff x\diff t\right|\leq C.
%\end{align}
Combination of \eqref{druha}, \eqref{firstlimitforepsilonbound}, \eqref{secondlimitforepsilonbound}, \eqref{cruciallimitforepsilonbound}, \eqref{thirdlimitforepsilonbound}
directly leads to
\begin{align*}%\label{boundsthirdterm}
\int_{\Omega_T}\left|A^n\cdot\nabla_x u^n_+\,T''_{k,\varepsilon}(u^n_+)\psi(x)\right|\diff x\diff t \leq C(k,\|\psi\|_{W^{1,\infty}(\Omega)}).
\end{align*}
In a very similar way, we can get the same result also for $u^n_{-}$ and therefore we have
\begin{align}\label{boundsthirdterm}
\int_{\Omega_T}\left|A^n\cdot\nabla_x u^n\,T''_{k,\varepsilon}(u^n)\psi(x)\right|\diff x\diff t \leq  C(k,\|\psi\|_{W^{1,\infty}(\Omega)}).
\end{align}
Thus, we may go back to \eqref{abc}, use  the H\"{o}lder inequality and insert the estimate~\eqref{boundsthirdterm} and let $\beta\to T_{-}$ to observe
\begin{equation}
\begin{split}
&\int_{\Omega_T}(A(t, x, \nabla_x T_{k}(u^n))- A(t, x, \nabla_x T_{k}(u^m))\cdot \nabla_x T_{\delta}(\omega^{n, m})\psi\diff x\diff t\\
&\leq \limsup_{\varepsilon \to 0_+}\int_{\Omega_T}A^m\cdot\nabla_x u^m\,T''_{k,\varepsilon}(u^m)T_{\delta}(\omega^{n, m}_\varepsilon)\psi - A^n\cdot\nabla_x u^n\,T''_{k,\varepsilon}(u^n)T_{\delta}(\omega^{n, m}_\varepsilon)\psi\diff x\diff t\\
&\qquad + C(k)\delta\|\psi\|_{W^{1,\infty}(\Omega)}\\
&\le \delta \|\psi\|_{L^{\infty}(\Omega)}\limsup_{\varepsilon\to 0_+}\int_{\Omega_T}|A^m\cdot\nabla_x u^m\,T''_{k,\varepsilon}(u^m)|+ |A^n\cdot\nabla_x u^n\,T''_{k,\varepsilon}(u^n)|\diff x\diff t\\
&\qquad + C(k)\delta\|\psi\|_{W^{1,\infty}(\Omega)}\\
&  \le C(k,\|\psi\|_{W^{1,\infty}(\Omega)})\delta.
\end{split}\label{app:limitoftheproductwithdelta}
\end{equation}
%
%With this, we may finally apply Fatou's lemma to \eqref{boundsbeforeconvergencewithepsilon}. Indeed, using \eqref{boundsfirstterm}, \eqref{boundssecondterm}, \eqref{boundsthirdterm}
%\begin{align}\label{app:limitoftheproductwithdelta}
%    \int_0^\beta\int_{\Omega}(A(t, x, \nabla_x T_{k}(u^n))- A(t, x, \nabla_x T_{k}(u^m))\cdot(\nabla_x T_{k}(u^n) - \nabla_x T_{k}(u^m))T'_{\delta}(\omega^{n, m})\psi(x)\diff x\diff t \leq C(k)\delta.
%\end{align}
Having this, we may proceed to the proof of~\eqref{app:limitoftheproductinl1}. First, we fix $\delta > 0$. Then,
\begin{equation*}
\begin{split}
    &\int_{\Omega_T}\left|(A(t, x, \nabla_x T_k(u^n)) - A(t, x, \nabla_x T_k(u^m)))\cdot(\nabla_x T_k(u^n) - \nabla_x T_k(u^m))\psi(x)\right|^\mu\diff x\diff t\\
    &=\int_{\{|T_k(u^n) - T_k(u^m)| < \delta\}}\hspace{-50pt}\left|(A(t, x, \nabla_x T_k(u^n)) - A(t, x, \nabla_x T_k(u^m)))\cdot \nabla_x T_{\delta}(\omega^{n, m}) \psi(x)\right|^\mu\diff x\diff t\\
    &+ \int_{\{|T_k(u^n) - T_k(u^m)| > \delta\}}\hspace{-50pt}\left|(A(t, x, \nabla_x T_k(u^n)) - A(t, x, \nabla_x T_k(u^m)))\cdot(\nabla_x T_k(u^n) - \nabla_x T_k(u^m))\psi(x)\right|^\mu\diff x\diff t\\
    &\le \left(\int_{\Omega_T}(A(t, x, \nabla_x T_k(u^n)) - A(t, x, \nabla_x T_k(u^m)))\cdot \nabla_x T_{\delta}(\omega^{n, m}) \psi(x)\diff x\diff t\right)^\mu\\
    &\quad + C(k,\psi)|\{|T_k(u^n) - T_k(u^m)| > \delta\}|\\
    &\le C(k,\psi)\left(\delta^{\mu} + |\{|T_k(u^n) - T_k(u^m)| > \delta\}|\right),
\end{split}
\end{equation*}
where we subsequently used the H\"{o}lder inequality, the estimate \eqref{app:limitoftheproductwithdelta}, and the uniform bound~\eqref{rem:remarkonptxbounds}.
Next, by  Lemma~\ref{app:almosteverywhereconvergence} we know that $T_{k}(u^n)\to T_k(u)$ in measure. Consequently, letting $n,m \to \infty$ in the inequality above, we see that
\begin{equation*}
\begin{split}
    &\lim_{n,m\to \infty}\int_{\Omega_T}\left|(A(t, x, \nabla_x T_k(u^n)) - A(t, x, \nabla_x T_k(u^m)))\cdot(\nabla_x T_k(u^n) - \nabla_x T_k(u^m))\psi(x)\right|^\mu\diff x\diff t\\
    &\le C(k,\psi)\delta^{\mu}.
\end{split}
\end{equation*}
Since $\delta>0$ is arbitrary, we in fact proved that \begin{equation*}
    \begin{split}
        &\lim_{n, m\to +\infty}\int_{\Omega_T}\left|(A(t, x, \nabla_x T_k(u^n)) - A(t, x, \nabla_x T_k(u^m)))\cdot(\nabla_x T_k(u^n) - \nabla_x T_k(u^m))\psi(x)\right|^\mu\diff x\diff t = 0,
    \end{split}
\end{equation*}
which is the desired limit in \eqref{app:limitoftheproductinl1}.
\end{proof}

\subsection{Identification of the weak limits in the nonlinearities} Now, our final goal is to identify~$B_k$ defined via~\eqref{weaklimitoperatortrunc} as $A(t, x, \nabla_x T_k(u))$. We use the monotone operator theory but first, due to low integrability, we need to select proper sets, which is done in the lemma below.
\begin{lem}\label{lem:monotonicitytrickneededinequality}
Let $u^n$ be as in Proposition~\ref{app:propositionofexistence}. Then, there exists a non-increasing sequence of sets $E_j\subset \Omega_T$ such that $\lim_{j\to\infty}|E_j| = 0$ and
\begin{align}
    \limsup_{n\to\infty}\int_{\Omega_T \setminus E_j}A(t, x, \nabla_x T_k(u^n))\cdot\nabla_x T_k(u^n)\psi(x)\diff x\diff t = \int_{\Omega_T\setminus E_j}B_k \cdot \nabla_x T_k(u)\psi(x)\diff x\diff t,
\end{align}
for any $\psi\in L^{\infty}(\Omega)$.
\end{lem}
\begin{proof}
Since we can chose $\psi$ in Lemma~\ref{lem:limitoftheproductinpowermu} arbitrarily, we know that up to the subsequence
\begin{align}\label{almosteverywhereconvergenceofvnm}
V_{n, m} := (A(t, x, \nabla_x T_k(u^n)) - A(t, x, \nabla_x T_k(u^m)))\cdot(\nabla_x T_k(u^n) - \nabla_x T_k(u^m)) \to 0 \text{ a.e. in }\Omega_T
\end{align}
as $n,m\to \infty$. Furthermore, by \eqref{app:energytruncl1} we know that sequences  $\{|A(t,x,\nabla_x T_k(u^n))|^{p'(t,x)}\}_{n\in \N}$,  $\{|\nabla_x T_k(u^n)|^{p(t,x)}\}_{n\in \N}$ and $\{V_{n,m}\}_{n,m\in\N}$ are bounded in $L^1(\Omega_T)$. Thus, by the Chacon biting lemma, there exists a non-increasing sequence of sets $E_j\subset \Omega_T$ such that $\lim_{j\to\infty}|E_j| = 0$ and the functions $V$, $\overline{A\cdot \nabla_x T_k(u)}$ belonging to $L^1(\Omega_T)$, such that up to the subsequence
\begin{equation}\label{opl}
\begin{aligned}
V_{n, m} &\rightharpoonup V &&\text{ weakly in }L^1(\Omega_T\setminus E_j),\\
A(t,x,\nabla_x T_k(u^n))\cdot \nabla_x T_k(u^n) &\rightharpoonup \overline{A\cdot \nabla_x T_k(u)}&&\text{ weakly in }L^1(\Omega_T\setminus E_j).
\end{aligned}
\end{equation}
In particular, by the Dunford--Pettis theorem, $\{V_{n, m}\}$ is equiintegrable in $L^1(\Omega_T\setminus E_j)$. This, together with \eqref{almosteverywhereconvergenceofvnm} and the Vitali convergence theorem, gives us
\begin{align}
    V_{n,m} \rightarrow 0 \text{ strongly in }L^1(\Omega_T\setminus E_j).
\end{align}
Hence, for any $\psi\in L^{\infty}(\Omega)$, we have
\begin{equation*}
\begin{split}
    0 &= \lim_{n\to \infty} \left(\lim_{m\to \infty} \int_{\Omega_T\setminus E_j} V_{n,m} \psi \diff x \diff t\right)\\
    &=\lim_{n\to\infty}\int_{\Omega_T\setminus E_j}A(t, x, \nabla_x T_k(u^n))\cdot \nabla_x T_k(u^n) \psi\diff x\diff t\\
    &\quad +\lim_{m\to \infty}\int_{\Omega_T\setminus E_j}A(t, x, \nabla_x T_k(u^m))\cdot \nabla_x T_k(u^m)\psi\diff x\diff t\\
    &-\quad \lim_{n\to \infty}\left(\lim_{m\to \infty}\int_{\Omega_T\setminus E_j}A(t, x, \nabla_x T_k(u^m))\cdot \nabla_x T_k(u^n)\psi\diff x\diff t\right)\\
    &-\quad \lim_{n\to \infty}\left(\lim_{m\to \infty}\int_{\Omega_T\setminus E_j}A(t, x, \nabla_x T_k(u^n))\cdot \nabla_x T_k(u^m)\psi\diff x\diff t\right)\\
    &= 2\int_{\Omega_T \setminus E_j}\overline{A\cdot \nabla_x T_k(u)}\psi(x)\diff x\diff t - 2\int_{\Omega_T\setminus E_j}B_k \cdot \nabla_x T_k(u)\psi(x)\diff x\diff t,
\end{split}
\end{equation*}
where we have used the weak convergence result \eqref{opl} to identify the first term and also  \eqref{weaklimitoperatortrunc} and \eqref{weaklimittrunaction} to converge with mixed terms, and we remembered that $v_k$ has been identified with $\nabla_x T_k(u)$ in Lemma \ref{app:almosteverywhereconvergence}. This finishes the proof.
\end{proof}

Having proven the lemma above, we may proceed with the identification of $B_k$.
\begin{lem}\label{lem:identificationofthelimit}
Let $u^n$ be as in the Proposition \ref{app:propositionofexistence} and $B_k$ be defined via \eqref{weaklimitoperatortrunc}, then
$$
B_k = A(t, x, \nabla_x T_k(u)) \text{ a.e. in }\Omega_T.
$$
\end{lem}
\begin{proof}
Fix an arbitrary $\eta\in L^\infty(\Omega_T; \R^d)$ and an arbitrary, non-negative $\psi \in L^{\infty}(\Omega)$. Then, by Assumption~\ref{intro:assumA_mono}
\begin{align}\label{someinequalityformonotonicity}
    \int_{\Omega_T\setminus E_j}(A(t, x, \nabla_x T_k(u^n)) - A(t, x, \eta))\cdot(\nabla_x T_k(u^n) - \eta)\psi(x)\diff x\diff t \geq 0.
\end{align}
Next, we use of \eqref{weaklimitoperatortrunc} and \eqref{weaklimittrunaction} to deduce
\begin{align*}
    \int_{\Omega_T\setminus E_j}A(t, x, \nabla_x T_k(u^n))\cdot\eta\,\psi(x)\diff x\diff t &\rightarrow \int_{\Omega_T\setminus E_j}B_k\cdot\eta\,\psi(x)\diff x\diff t,\\
    \int_{\Omega_T\setminus E_j}A(t, x, \eta)\cdot\nabla_x T_k(u^n)\,\psi(x)\diff x\diff t &\rightarrow  \int_{\Omega_T\setminus E_j}A(t, x, \eta)\cdot\nabla_x T_k(u)\,\psi(x)\diff x\diff t.
\end{align*}
Thus, with the use of Lemma \ref{lem:monotonicitytrickneededinequality} and inequality \eqref{someinequalityformonotonicity} we arrive at
\begin{align*}
    \int_{\Omega_T\setminus E_j}(B_k - A(t, x, \eta))\cdot(\nabla_x T_k(u) - \eta)\psi(x)\diff x\diff t \geq 0.
\end{align*}
Using Lemma \ref{res:monot_trick} we immediately get
$$
B_k = A(t, x, \nabla_x T_k(u)) \text{ a.e. in }\Omega_T\setminus E_j.
$$
Finally, since $\lim_{j\to \infty}|E_j| = 0$, we arrive at the needed conclusion.
\end{proof}

\subsection{Existence of a weak solution}
In this part we deal with the assumptions on the lower bound of the function $p$, and we want to let $n\to \infty$ in \eqref{app:weakformulation} to obtain \eqref{wfweak} and \eqref{wfweakII}, which finishes the proof of the existence of weak solution.
Using Remark~\ref{rem:boundednessofnalbaxuandoperator}, we see that for a subsequence
\begin{equation}
\begin{aligned}
A^n &\rightharpoonup \overline{A} &&\text{ weakly in }L^a(\Omega_T; \mathbb{R}^d),\\
u^n &\rightharpoonup u &&\text{ weakly in }L^a(\Omega_T),
\end{aligned}\label{wcre}
\end{equation}
for some $a>1$. Note that at this point we have only used the assumption $p_{\min}>\frac{2d}{d+1}$. In case we assume that $p_{\min}>\frac{2d+1}{d+1}$, we can obtain
\begin{equation*}
\begin{aligned}
\nabla_x u^n &\rightharpoonup \nabla_x u &&\text{ weakly in }L^a(\Omega_T; \mathbb{R}^d).
\end{aligned}
\end{equation*}
Therefore, we can use the properties of $f^n$ and $u_0^n$ stated in~\eqref{app:strongconvergenceoffM}--\eqref{app:strongconvergenceofboundarydataM} and let $n\to \infty$ in \eqref{app:weakformulation} to conclude that
\begin{align}\label{wfweakb}
    \int_{\Omega_T}-u\,\p_t\phi(t, x)\diff x\diff t + \int_{\Omega_T}\overline{A} \cdot \nabla_x \phi\diff x\diff t = \langle f,\phi\rangle_{\mathcal{M}(\Omega_T)} + \langle u_0, \phi(0, \cdot)\rangle_{\mathcal{M}(\Omega)}
\end{align}
for all $\phi\in C^\infty_c((-\infty, T)\times\Omega)$. It remains to show that
\begin{equation}
\label{AequalA}
\overline{A}=A(t,x,\nabla_x u) \textrm{ a.e. in } \Omega_T.
\end{equation}
To do so, we use the weak lower semi-continuity of norms and the weak convergence results \eqref{wcre} and  \eqref{weaklimitoperatortrunc} together with the identification stated in Lemma~\ref{lem:identificationofthelimit} and observe that
$$
\begin{aligned}
&\int_{\Omega_T}|\overline{A}-A(t,x,\nabla_x T_k(u))|\diff x \diff t \le  \liminf_{n\to \infty}\int_{\Omega_T}|A(t,x,\nabla_x u^n)-A(t,x,\nabla_x T_k(u^n))|\diff x \diff t\\
&= \liminf_{n\to \infty}\int_{\Omega_T}|A(t,x,\nabla_x u^n)|\chi_{\{|u^n|>k\}}\diff x \diff t \le  \liminf_{n\to \infty}\|A^n\|_{L^a(\Omega_T)}|\{|u^n|>k\}|^{\frac{a-1}{a}}\\
&\le  \liminf_{n\to \infty}\frac{C}{k^{\frac{a-1}{a}}}.
\end{aligned}
$$
Letting $k\to \infty$ in the inequality above, we deduce
\begin{align}\label{kilop}
\lim_{k\to \infty}\int_{\Omega_T}|\overline{A}-A(t,x,\nabla_x T_k(u))|\diff x \diff t =0.
\end{align}
Obviously, from the relation above, we see that
\begin{equation}
\label{identp}
\overline{A}=A(t,x,\nabla_x T_k(u)) \quad \textrm{a.e. on the set } \{|u|\le k\}.
\end{equation}
Thus in case $p_{\min}>\frac{2d+1}{d+1}$, we know that $\nabla_x u\in L^a(\Omega_T)$ and \eqref{identp} leads to
$$
\overline{A}=A(t,x,\nabla_x u) \quad \textrm{a.e. in } \Omega_T,
$$
which finishes the proof of the existence of a weak solution as in Definition \ref{def:weaksolution}. In addition, the identification \eqref{identp} allows us to define the notion of a weak solution for $p_{\min}\in (2d/(d+1), (2d+1)/(d+1)]$ in the sense of Definition~\ref{wierd}.

\subsection{Existence of an entropy solution}
To show the existence of an entropy solution, we still need to verify \eqref{ineq:inequalityinentropydefinition}. To see it, notice that following the proof of \cite[ Lemma 4.2]{bulicek2021parabolic} we can see that for~$u^n$ found in Proposition~\ref{app:propositionofexistence}, the identity
\begin{equation}
\begin{split}
    &\int_{\Omega}G_k(u^n(t, x) - \phi(t, x)) - G_k(u^n_0(x) - \phi(0, x))\diff x + \int_0^t\int_{\Omega}T_k(u^n - \phi)\,\p_t\phi(t, x)\diff x\diff \tau\\
    &\phantom{=} + \int_0^t\int_{\Omega}A(t, x, \nabla_x u^n)\cdot\nabla_x T_k(u^n - \phi)\diff x\diff \tau\\
    &= \int_0^t\int_{\Omega}f^n\,T_k(u^n - \phi)\diff x\diff \tau,
\end{split}
\end{equation}
holds true for any $k\in \R_+$ and $\phi\in C^\infty_0((-\infty, T)\times\Omega)$. We need to pass to the limit in all terms to deduce~\eqref{ineq:inequalityinentropydefinition}.

First, thanks to Lemma~\ref{app:almosteverywhereconvergence}, we have
\begin{align}\label{app:strongconvergencetruncationwithphi}
    T_k(u^n - \phi) \rightarrow T_k(u - \phi)\quad \text{strongly in }L^1(\Omega_T).
\end{align}
Consequently,
$$
\lim_{n\to \infty} \int_0^t\int_{\Omega}T_k(u^n - \phi)\,\p_t\phi(t, x)\diff x\diff \tau = \int_0^t\int_{\Omega}T_k(u - \phi)\,\p_t\phi(t, x)\diff x\diff \tau.
$$
Similarly, having \eqref{app:strongconvergenceoff}, \eqref{app:strongconvergenceofboundarydata} and using also \eqref{app:strongconvergencetruncationwithphi} and the boundedness of $T_k$ and the definition of $G_k$ we get
$$
\begin{aligned}
    \lim_{n\to \infty} \int_{\Omega} G_k(u^n_0(x) - \phi(0, x))\diff x &=\int_{\Omega}G_k(u_0(x) - \phi(0, x))\diff x \\
    \lim_{n\to \infty} \int_0^t\int_{\Omega}f^n\,T_k(u^n - \phi)\diff x\diff \tau&=\int_0^t\int_{\Omega}f\,T_k(u - \phi)\diff x\diff \tau.
\end{aligned}
$$
In addition, using the fact that $G_k$ is bounded from below and also the pointwise convergence of~$u^n$, see \eqref{pointA}, we may use the Fatou lemma to conclude
$$
\liminf_{n\to \infty} \int_{\Omega}G_k(u^n(t, x) - \phi(t, x))\diff x \ge \int_{\Omega}G_k(u(t, x) - \phi(t, x)) \diff x
$$
for almost all $t\in (0,T)$.

Hence, it remains to show that
%\begin{lem}
%Let $u^n$ be as in Proposition %\ref{app:propositionofexistence}, then
\begin{align}\label{existence_of_entropy_sol_inequality}
    \liminf_{n\to\infty}\int_0^t\int_{\Omega}A(t, x, \nabla_x u^n)\cdot\nabla_x T_k(u^n - \phi)\diff x\diff \tau\geq \int_0^t\int_{\Omega}A(t, x, \nabla_x u)\cdot\nabla_x T_k(u - \phi)\diff x\diff \tau.
\end{align}
%for any $\phi\in C^\infty_c([0, T]\times\Omega)$ and %almost every $t\in (0, T)$.
%\end{lem}
%\begin{proof}
The rest of this section is devoted to the proof of~\eqref{existence_of_entropy_sol_inequality}. In fact, we prove \eqref{existence_of_entropy_sol_inequality} not only for smooth functions $\phi$ but also for functions $\phi$ satisfying $\phi\in L^{\infty}(\Omega_T)\cap L^1(0,T; W^{1,1}_0(\Omega))$ fulfilling in addition $\nabla_x \phi \in L^{p(t,x)}(\Omega_T;\mathbb{R}^d)$.

Due to the Egoroff theorem and the convergence result \eqref{app:strongconvergencetruncationwithphi}, we know that  for any $\sigma > 0$, there exists a measurable set $E_{\sigma}\subset \Omega_T$, such that $|\Omega_T \setminus E_{\sigma}| < \sigma$ and
\begin{align}\label{app:uniformconvergencewithphi}
    T_k(u^n - \phi) \rightarrow T_k(u - \phi)\quad\text{ uniformly in }E_{\sigma}.
\end{align}
Next, after we denote
$$
E_{\delta} := \{(t, x)\in \Omega_T \quad :\quad |u(t, x) - \phi(t, x)| < k - \delta\},
$$
we can conclude with the use of \eqref{app:uniformconvergencewithphi}, that there exists $n_0$, such that for all $n\geq n_0$
$$
T_k(u^n - \phi) = u^n - \phi
$$
holds in $E_{\sigma}\cap E_{\delta}$. Hence, defining $M:= \|\phi\|_{\infty}$ and using \ref{intro:assumA_mono}, \ref{intro:assumA_vanish} we obtain (recall that $\chi$ denotes the characteristic function of the corresponding set, and also recall that $E_j$ were constructed in Lemma~\ref{lem:monotonicitytrickneededinequality})
\begin{equation}\label{dosad}
\begin{aligned}
    \int_0^t\int_{\Omega}&A(\tau, x, \nabla_x u^n)\cdot\nabla_x T_k(u^n - \phi)\diff x\diff \tau =\int_0^t\int_{\Omega}A(\tau, x, \nabla_x \phi)\cdot\nabla_x T_k(u^n - \phi)\diff x\diff \tau\\
    &+\int_0^t\int_{\Omega}\left(A(\tau, x, \nabla_x u^n)-A(\tau, x, \nabla_x \phi)\right)\cdot\nabla_x T_k(u^n - \phi)\diff x\diff \tau\\
    &\geq \int_0^t\int_{\Omega}\chi_{(\Omega_T\setminus E_j)\cap E_{\sigma}\cap E_{\delta}}\left(A(\tau, x, \nabla_x u^n)-A(\tau, x, \nabla_x \phi)\right)\cdot\nabla_x T_k(u^n - \phi)\diff x\diff \tau\\
    &\qquad +\int_0^t\int_{\Omega}A(\tau, x, \nabla_x \phi)\cdot\nabla_x T_k(u^n - \phi)\diff x\diff \tau\\
    &\overset{n>n_0}= \int_0^t\int_{\Omega}\chi_{(\Omega_T\setminus E_j)\cap E_{\sigma}\cap E_{\delta}}\left(A(\tau, x, \nabla_x T_{M+k}(u^n))-A(\tau, x, \nabla_x T_M(\phi))\right)\cdot \\
    &\qquad {} \qquad {}\qquad {}\qquad {} \cdot\nabla_x (T_{M+k}(u^n) - T_M(\phi))\diff x\diff \tau\\
    &\qquad +\int_0^t\int_{\Omega}A(\tau, x, \nabla_x \phi)\cdot\nabla_x T_k(u^n - \phi)\diff x\diff \tau.
\end{aligned}
\end{equation}
%where $E_j$ is defined as in the proof of the Lemma \ref{lem:monotonicitytrickneededinequality}.
It is rather standard to identify the limit in the last term on the right-hand side and with the help of the growth assumption~\ref{intro:ass_on_A:coercgr}, the convergence results~\eqref{weaklimittrunaction} and \eqref{pointA}, we see that
\begin{equation}\label{easyguy}
\lim_{n\to \infty}\int_0^t\int_{\Omega}A(\tau, x, \nabla_x \phi)\cdot\nabla_x T_k(u^n - \phi)\diff x\diff \tau=\int_0^t\int_{\Omega}A(\tau, x, \nabla_x \phi)\cdot\nabla_x T_k(u - \phi)\diff x\diff \tau.
\end{equation}
For the first term on the right-hand side of~\eqref{dosad}, we use~Lemma~\ref{lem:monotonicitytrickneededinequality} and Lemma~\ref{lem:identificationofthelimit} to get
\begin{equation}\label{hardguy}
\begin{aligned}
\lim_{n\to \infty}&\int_0^t\int_{\Omega}\chi_{(\Omega_T\setminus E_j)\cap E_{\sigma}\cap E_{\delta}}\left(A(\tau, x, \nabla_x T_{M+k}(u^n))-A(\tau, x, \nabla_x T_M(\phi))\right)\cdot \\
    &\qquad {} \qquad {}\qquad {}\qquad {} \cdot\nabla_x (T_{M+k}(u^n) - T_M(\phi))\diff x\diff \tau\\
    &=\int_0^t\int_{\Omega}\chi_{(\Omega_T\setminus E_j)\cap E_{\sigma}\cap E_{\delta}}\left(A(\tau, x, \nabla_x T_{M+k}(u))-A(\tau, x, \nabla_x T_M(\phi))\right)\cdot \\
    &\qquad {} \qquad {}\qquad {}\qquad {} \cdot\nabla_x (T_{M+k}(u) - T_M(\phi))\diff x\diff \tau\\
    &=\int_0^t\int_{\Omega}\chi_{(\Omega_T\setminus E_j)\cap E_{\sigma}}\left(A(\tau, x, \nabla_x u)-A(\tau, x, \nabla_x \phi)\right)\cdot \nabla_x T_{k-\delta} (u - \phi)\diff x\diff \tau.
    \end{aligned}
\end{equation}
Consequently, letting $n\to \infty$ in \eqref{dosad} and using \eqref{easyguy}--\eqref{hardguy}, we obtain
\begin{align*}
    \liminf_{n\to\infty}&\int_0^t\int_{\Omega}A(\tau, x, \nabla_x u^n)\cdot\nabla_x T_k(u^n - \phi)\diff x\diff \tau\\
    &\geq \int_0^t\int_{\Omega}\chi_{(\Omega_T\setminus E_j)\cap E_{\sigma}}\left(A(\tau, x, \nabla_x u)-A(\tau, x, \nabla_x \phi)\right)\cdot \nabla_x T_{k-\delta} (u - \phi)\diff x\diff \tau\\
    &\qquad +\int_0^t\int_{\Omega}A(\tau, x, \nabla_x \phi)\cdot\nabla_x T_k(u - \phi)\diff x\diff \tau.
\end{align*}
Since $\phi$ is bounded, we can use the Lebesgue dominated convergence theorem and let $\sigma \to 0_+$, $\delta\to 0_+$ and $j\to\infty$ in the inequality above to get~\eqref{existence_of_entropy_sol_inequality}. This finishes the proof of the existence of an entropy solution.
%\end{proof}

\section{Uniqueness of an entropy solution}\label{S4}
%\begin{proof}
We proceed here using the method introduced in \cite{MR1436364}, which we must, however, rebuild here, due to the dependence of the exponent $p(t,x)$ on the spatial and the time variable. Let $u^1$ be an arbitrary entropy solution according to the Definition~\ref{def:entropysolution}. We will show that it coincides with a solution $u$, constructed in the previous section, as the limit of the approximation $u^n$. To this end, notice that similarly to $u^n$ in Lemma \ref{lem:extension_by_0_u} one can deduce from \eqref{eq:renormalizedequationforun} that
$$
(T_{m,\varepsilon}(\overline{u^n}))^\kappa \in W^{1, 1}((-T, T)\times\Omega') \quad\text{ for any }\Omega'\Subset \Omega,
$$
where $\overline{u^n}$ is defined as in Lemma \ref{lem:extension_by_0_u}, $T_{m,\varepsilon}$ as in \eqref{def:trunc_smooth}, superscript $\kappa$ is connected to the mollification in the spatial variable (see the beginning of the Section \ref{S3}), and $m\in \mathbb{N}$, $\varepsilon\in (0,1)$ and $\kappa\in (0,1)$ are arbitrary. Next, we consider a smooth non-negative compactly supported $\psi\in C_c^\infty(\Omega)$ fulfilling $0\leq \psi \leq 1$. In fact, we require that $\psi$ is supported in $\Omega_{\kappa}$, where
$$
\Omega_{\kappa}:=\{x\in \Omega; \; B_{2\kappa}(x)\subset \Omega\}.
$$
Then, we use \eqref{ineq:inequalityinentropydefinition} for $u^1$ with the setting $\phi:=\mathcal{R}^\alpha((T_{m,\varepsilon}(\overline{u^n}))^\kappa)\psi$ (here, recall that $\mathcal{R}^\alpha$ is connected with mollification in the time variable - see the beginning of the Section \ref{S3}). The resulting  inequality reads as
\begin{equation}\label{ineq:inequalityinentropydefinition2}
\begin{split}
    &\int_{\Omega}G_k(u^1(t) - \mathcal{R}^\alpha((T_{m,\varepsilon}(\overline{u^n}))^\kappa)(t)\psi) - G_k(u_0(x) - \mathcal{R}^\alpha((T_{m,\varepsilon}(\overline{u^n}))^\kappa)(0)\psi)\diff x \\
    &\quad + \int_0^t\int_{\Omega}T_k(u^1 - \mathcal{R}^\alpha((T_{m,\varepsilon}(\overline{u^n}))^\kappa)\psi)\,\p_t\mathcal{R}^\alpha((T_{m,\varepsilon}(\overline{u^n}))^\kappa)\psi\diff x\diff \tau\\
    &\phantom{=} + \int_0^t\int_{\Omega}A(t, x, \nabla_x u^1)\cdot\nabla_x T_k(u^1 - \mathcal{R}^\alpha((T_{m,\varepsilon}(\overline{u^n}))^\kappa)\psi)\diff x\diff \tau \\
    &\leq \int_0^t\int_{\Omega}f\,T_k(u^1 - \mathcal{R}^\alpha((T_{m,\varepsilon}(\overline{u^n}))^\kappa)\psi)\diff x\diff \tau.
\end{split}
\end{equation}
Subsequently, we use the renormalized identity ~\eqref{eq:renormalizedequationforun} (with $k$ replaced by $m$), multiply it by $-(\mathcal{R}^\alpha(T_k(u^1 - \mathcal{R}^\alpha((T_{m,\varepsilon}(\overline{u^n}))^\kappa)\psi))\psi)^\kappa$ and integrate over $\Omega$ and $(0,t)$ to get, after integration by parts,
\begin{align}
\begin{aligned}\label{eq:renormalizedequationforu2}
-&\int_0^t \int_{\Omega}\p_t T_{m,\varepsilon}(u^n)(\mathcal{R}^\alpha(T_k(u^1 - \mathcal{R}^\alpha((T_{m,\varepsilon}(\overline{u^n}))^\kappa)\psi))\psi)^\kappa \diff x \diff \tau \\
&= \int_0^t \int_{\Omega} A(\tau, x, \nabla_x u^n)T'_{m,\varepsilon}(u^n)\cdot \nabla_x(\mathcal{R}^\alpha(T_k(u^1 - \mathcal{R}^\alpha((T_{m,\varepsilon}(\overline{u^n}))^\kappa)\psi))\psi)^\kappa  \diff x \diff \tau\\
&\quad - \int_0^t \int_{\Omega}f^n\, T'_{m,\varepsilon}(u^n)(\mathcal{R}^\alpha(T_k(u^1 - \mathcal{R}^\alpha((T_{m,\varepsilon}(\overline{u^n}))^\kappa)\psi))\psi)^\kappa \diff x \diff \tau\\
&\quad + \int_0^t \int_{\Omega} A(\tau, x,\nabla_x u^n)\cdot\nabla_x(T'_{m,\varepsilon}(u^n))(\mathcal{R}^\alpha(T_k(u^1 - \mathcal{R}^\alpha((T_{m,\varepsilon}(\overline{u^n}))^\kappa)\psi))\psi)^\kappa.
\end{aligned}
\end{align}
Using the definition of $\mathcal{R}^{\alpha}$ and also using the fact that the superscript $\kappa$ denotes the standard mollification by convolution, we see that the second term on the left-hand side of \eqref{ineq:inequalityinentropydefinition2} and the term on the left-hand side of \eqref{eq:renormalizedequationforu2} are equal up to the sign. Thus, summing \eqref{ineq:inequalityinentropydefinition2}  and \eqref{eq:renormalizedequationforu2} we get the starting inequality (we have also used the definition of $\overline{u^n}$)
\begin{equation}\label{proof:uniquenessfirstinequality}
    \begin{split}
        &\int_0^t\int_{\Omega} A(\tau, x, \nabla_x u^n)\cdot\nabla_x T_k(u^1 - \mathcal{R}^\alpha((T_{m,\varepsilon}(\overline{u^n}))^\kappa)\psi)\diff x\diff\tau\\
        &\phantom{=} - \int_0^t\int_{\Omega}A(\tau, x, \nabla_x u^n)T'_{m,\varepsilon}(u^n)\cdot \nabla_x (\mathcal{R}^\alpha(T_k(u^1 - \mathcal{R}^\alpha((T_{m,\varepsilon}(\overline{u^n}))^\kappa)\psi))\psi)^\kappa\diff x\diff \tau\\
        &\phantom{=} \int_{\Omega}G_k(u^1(t, x) - \mathcal{R}^\alpha((T_{m,\varepsilon}(\overline{u^n}(t, x)))^\kappa)\psi(x))\diff x\\
        &\leq \int_0^t\int_{\Omega}f\,T_k(u^1 - \mathcal{R}^\alpha((T_{m,\varepsilon}(\overline{u^n}))^\kappa)\psi)\diff x\diff\tau\\
        &\phantom{=}-\int_0^t\int_{\Omega}f^n\,T'_{m,\varepsilon}(u^n)(\mathcal{R}^\alpha(T_k(u^1 - \mathcal{R}^\alpha((T_{m,\varepsilon}(\overline{u^n}))^\kappa)\psi))\psi)^\kappa \diff x\diff \tau\\
        &\phantom{=}+\int_0^t\int_{\Omega}A(\tau, x,\nabla_x u^n)\cdot\nabla_x (T'_{m,\varepsilon}(u^n))(\mathcal{R}^\alpha(T_k(u^1 - \mathcal{R}^\alpha((T_{m,\varepsilon}(\overline{u^n}))^\kappa)\psi))\psi)^\kappa\diff x\diff \tau\\
        &\phantom{=}+\int_{\Omega}G_k(u_0(x) - \mathcal{R}^\alpha((T_{m,\varepsilon}(\overline{u^n}(0,x)))^\kappa)\psi(x))\diff x.
    \end{split}
\end{equation}
Our aim is to converge with $\alpha \to 0_+$, $\kappa\to 0_+$, $\psi\to 1$, $\varepsilon\to 0_+$ and $n\to \infty$ in this order, and we want to show that the right-hand side vanishes. Furthermore, due to the monotonicity of $A$ and the definition of $G_k$, all integrals on the left-hand side will be non-negative, which will, in the end, provide the uniqueness.
%Therefore let us divide the proof into the six steps.\\
%\\
\subsection{Step 1: convergence with \texorpdfstring{$\alpha \to 0_+$}{a}} Here, trivially, all the terms in the inequality \eqref{proof:uniquenessfirstinequality} can be treated by the Lebesgue dominated convergence theorem. Indeed, since in all the terms there are mollifications with respect to the spatial variable denoted by $\kappa$ and we know that $u^n$ and $u^1$ belong to the $L^{\infty}((0,T); L^1(\Omega))$, the use of the Lebesgue theorem becomes easy. Hence, we skip the details and get
\begin{equation}\label{proof:uniquenesssecondinequality}
    \begin{split}
        &\int_0^t\int_{\Omega} A(\tau, x, \nabla_x u^1)\cdot\nabla_x T_k(u^1 - (T_{m,\varepsilon}(u^n))^\kappa\psi)\diff x\diff\tau\\
        &\phantom{=} - \int_0^t\int_{\Omega}A(\tau, x, \nabla_x u^n)T'_{m,\varepsilon}(u^n)\cdot \nabla_x (T_k(u^1 - (T_{m,\varepsilon}(u^n))^\kappa\psi)\psi)^\kappa\diff x\diff \tau\\
        &\phantom{=}+\int_{\Omega}G_k(u^1(t, x) - (T_{m,\varepsilon}(u^n(t, x)))^\kappa\psi(x))\diff x\\
        &\leq \int_0^t\int_{\Omega}f\,T_k(u^1 - (T_{m,\varepsilon}(u^n))^\kappa\psi)\diff x\diff\tau\\
        &\phantom{=}-\int_0^t\int_{\Omega}f^n\,T'_{m,\varepsilon}(u^n)(T_k(u^1 - (T_{m,\varepsilon}(u^n))^\kappa\psi)\psi)^\kappa \diff x\diff \tau\\
        &\phantom{=}+\int_0^t\int_{\Omega}A(\tau, x,\nabla_x u^n)\cdot\nabla_x (T'_{m,\varepsilon}(u^n))(T_k(u^1 - (T_{m,\varepsilon}(u^n))^\kappa\psi)\psi)^\kappa\diff x\diff \tau\\
        &\phantom{=}+\int_{\Omega}G_k(u_0(x) - (T_{m,\varepsilon}(u_0^n(x)))^\kappa\psi(x))\diff x.
    \end{split}
\end{equation}
%\\
\subsection{Step 2: convergence with \texorpdfstring{$\kappa\to 0_+$}{k}}
The right-hand side of the inequality \eqref{proof:uniquenesssecondinequality} might be again treated easily with the Lebesgue dominated convergence theorem by using the fact that $T_k$~is bounded. In a similar manner, we can pass to the limit in the third term on the left-hand side for almost all $t\in (0,T)$. Thus, we focus on the first two terms on the left-hand side of \eqref{proof:uniquenesssecondinequality}. For the first term we can write
\begin{align*}
    \int_0^t\int_{\Omega} A(\tau, x, \nabla_x u^1)&\cdot\nabla_x T_k(u^1 - (T_{m,\varepsilon}(u^n))^\kappa\psi)\diff x\diff\tau = \\
    &=\int_0^t\int_{\Omega} A(\tau, x, \nabla_x T_{k+m+\varepsilon}(u^1))\cdot\nabla_x T_k(T_{k+m+\varepsilon}(u^1) - (T_{m,\varepsilon}(u^n))^\kappa\psi)\diff x\diff\tau.
\end{align*}
Since $u^1$ is an entropy solution, we know that
$$
A(\tau, x, \nabla_x T_{k+m+\varepsilon}(u^1))\in L^{p'(t,x)}(\Omega_T; \R^d).
$$
Furthermore, because $\psi$ is compactly supported, we may apply  Proposition~\ref{prop:convergence_of_convolution_in_musielak_orlicz} to obtain
$$
\nabla_x (T_{m,\varepsilon}(u^n))^\kappa\psi \to \nabla_x (T_{m,\varepsilon}(u^n))\psi \textrm{ modularly in }  L^{p'(t,x)}(\Omega_T;\R^d).
$$
Therefore, we can now use Theorem~\ref{thm:modular_l1} to deduce
%
%\begin{align*}
%    &= \int_0^t\int_{\Omega}\left(A(\tau, x,\nabla_x u^1) \cdot\nabla_x u^1\right) \,T'_k(u^1 - (T_{m,\varepsilon}(u^n))^\kappa\psi)\diff x\diff \tau\\
%    &\phantom{=}- \int_{0}^t\int_{\Omega}\left(A(\tau, x,\nabla_x u^1)\cdot\nabla_x (T_{m,\varepsilon}(u^n))^\kappa\right)\, T'_k(u^1 - (T_{m,\varepsilon}(u^n))^\kappa\psi)\,\psi\diff x\diff \tau\\
%    &\phantom{=}- \int_0^t\int_{\Omega}\left(A(\tau, x,\nabla_x u^1)\cdot \nabla_x\psi\right)\, T'_k(u^1 - (T_{m,\varepsilon}(u^n))^\kappa\psi)\, (T_{m,\varepsilon}(u^n))^\kappa\diff x\diff\tau.
%\end{align*}
%First summand is treatable by Lebesgue's dominated convergence theorem, while the second and the third summand by Proposition \ref{prop:convergence_of_convolution_in_musielak_orlicz}. We obtain
\begin{equation}
\begin{split}\label{proof:uniqueness_kappa_1}
    \lim_{\kappa\to 0_+}\int_0^t\int_{\Omega} A(\tau, x, \nabla_x u^1)\cdot\nabla_x T_k(u^1 - (T_{m,\varepsilon}(u^n))^\kappa\psi)\diff x\diff\tau \\
    = \int_0^t\int_{\Omega} A(\tau, x, \nabla_x u^1)\cdot\nabla_x T_k(u^1 - T_{m,\varepsilon}(u^n)\psi)\diff x\diff\tau.
\end{split}
\end{equation}
The second term on the left-hand side of \eqref{proof:uniquenesssecondinequality} can be rewritten as
$$
\begin{aligned}
\int_0^t\int_{\Omega}&A(\tau, x, \nabla_x u^n)T'_{m,\varepsilon}(u^n)\cdot \nabla_x (T_k(u^1 - (T_{m,\varepsilon}(u^n))^\kappa\psi)\psi)^\kappa\diff x\diff \tau=\\
&=\int_0^t\int_{\Omega}\left(A(\tau, x, \nabla_x T_{m+\varepsilon}(u^n))T'_{m,\varepsilon}(u^n)\right)^{\kappa}\cdot \nabla_x T_k(T_{m+\varepsilon+k}(u^1) - (T_{m,\varepsilon}(u^n))^\kappa\psi)\psi\diff x\diff \tau.
\end{aligned}
$$
Thus, we can argue very similarly as in the previous term and with the help of~Proposition~\ref{prop:convergence_of_convolution_in_musielak_orlicz} and Theorem~\ref{thm:modular_l1} we can conclude
%
%For the second term, we again split it into the smaller parts. We have
%\begin{align*}
 %   \int_0^t\int_{\Omega}A(\tau, &x, \nabla_x T_{m,\varepsilon}(u^n))\cdot \nabla_x (T_k(u^1 - (T_{m,\varepsilon}(u^n))^\kappa\psi)\psi)^\kappa\diff x\diff \tau =\\
  %  &= \int_0^t\int_{\Omega}\left(\left(A(\tau, x, \nabla_x T_{m,\varepsilon}(u^n))\right)^\kappa\cdot \nabla_x\psi\right)\,T_k(u^1 - (T_{m,\varepsilon}(u^n))^\kappa\psi)\diff x\diff \tau\\
  %  &\phantom{=} + \int_0^t\int_{\Omega}\left(\left(A(\tau, x, \nabla_x T_{m,\varepsilon}(u^n))\right)^\kappa\cdot \nabla_x u^1\right)\,T'_k(u^1 - (T_{m,\varepsilon}(u^n))^\kappa\psi)\,\psi\diff x\diff \tau\\
   % &\phantom{=} - \int_0^t\int_{\Omega}\left(\left(A(\tau, x, \nabla_x T_{m,\varepsilon}(u^n))\right)^\kappa\cdot \nabla_x \psi\right)\,T'_k(u^1 - (T_{m,\varepsilon}(u^n))^\kappa\psi)\,(T_{m,\varepsilon}(u^n))^\kappa\,\psi\diff x\diff \tau\\
    %&\phantom{=} - \int_0^t\int_{\Omega}\left(\left(A(\tau, x, \nabla_x T_{m,\varepsilon}(u^n))\right)^\kappa\cdot \nabla_x (T_{m,\varepsilon}(u^n))^\kappa\right)\,T'_k(u^1 - (T_{m,\varepsilon}(u^n))^\kappa\psi)\,\psi^2\diff x\diff \tau.\\
%\end{align*}
%Now we deal with first and second summand by using Proposition \ref{prop:convergence_of_convolution_in_musielak_orlicz}. Furthermore, third and fourth are treatable by both Proposition \ref{prop:convergence_of_convolution_in_musielak_orlicz} and Theorem \ref{thm:modular_l1}. We deduce
\begin{equation}
\begin{split}\label{proof:uniqueness_kappa_2}
\lim_{\kappa\to 0_+}\int_0^t\int_{\Omega}A(\tau, x, \nabla_x u^n)T'_{m,\varepsilon}(u^n)\cdot \nabla_x (T_k(u^1 - (T_{m,\varepsilon}(u^n))^\kappa\psi)\psi)^\kappa\diff x\diff \tau =\\
= \int_0^t\int_{\Omega}A(\tau, x, \nabla_x u^n)T'_{m,\varepsilon}(u^n)\cdot \nabla_x (T_k(u^1 - T_{m,\varepsilon}(u^n)\psi)\psi)\diff x\diff \tau.
\end{split}
\end{equation}
Combining both \eqref{proof:uniqueness_kappa_1} and \eqref{proof:uniqueness_kappa_2} we obtain from \eqref{proof:uniquenesssecondinequality}
\begin{equation}\label{proof:uniquenessthirdinequality}
    \begin{split}
        &\int_0^t\int_{\Omega} A(\tau, x, \nabla_x u^1)\cdot\nabla_x T_k(u^1 - T_{m,\varepsilon}(u^n)\psi)\diff x\diff\tau\\
        &\phantom{=} - \int_0^t\int_{\Omega}A(\tau, x, \nabla_x u^n)T'_{m,\varepsilon}(u^n)\cdot \nabla_x (T_k(u^1 - T_{m,\varepsilon}(u^n)\psi)\psi)\diff x\diff \tau\\
        &\phantom{=}+\int_{\Omega}G_k(u^1(t, x) - T_{m,\varepsilon}(u^n(t, x))\psi)\diff x\\
        &\leq \int_0^t\int_{\Omega}f\,T_k(u^1 - T_{m,\varepsilon}(u^n)\psi)\diff x\diff\tau\\
        &\phantom{=}-\int_0^t\int_{\Omega}f^n\,T'_{m,\varepsilon}(u^n)\,T_k(u^1 - T_{m,\varepsilon}(u^n)\psi)\,\psi \diff x\diff \tau\\
        &\phantom{=}+\int_0^t\int_{\Omega}A(\tau, x,\nabla_x u^n)\cdot\nabla_x (T'_{m,\varepsilon}(u^n))(T_k(u^1 - T_{m,\varepsilon}(u^n)\psi)\psi)\diff x\diff \tau\\
        &\phantom{=}+\int_{\Omega}G_k(u_0(x) - T_{m,\varepsilon}(u_0^n(x))\psi)\diff x.
    \end{split}
\end{equation}
%\\
\subsection{Step 3: convergence with \texorpdfstring{ $\psi \nearrow 1$}{p}} Our goal is to pass to the limit in \eqref{proof:uniquenessthirdinequality} with $\psi \nearrow 1$. To do so, we consider a special sequence $\{\psi_j\}_{j\in \mathbb{N}}$ found in Lemma~\ref{proof:uniqueness:crucial_lemma}, for which we know that $\nabla_x \psi_j$ is supported in $\Omega\setminus\Omega_j$, which satisfies $|\Omega\setminus\Omega_j|\to 0$ as $j\to \infty$. Since $\psi_j$ converges to $1$ as $j\to\infty$, it is easy to see, that all of the terms in \eqref{proof:uniquenessthirdinequality} that do not contain $\nabla_x\psi_j$ converge to proper limits using the Lebesgue dominated convergence theorem. Hence, let us focus our attention on the terms that do contain $\nabla_x \psi_j$. There are three such terms. The first one is of the form
\begin{align*}
    \int_0^t\int_{\Omega} (A(\tau, x, \nabla_x u^1)\cdot\nabla_x \psi_j) \,T'_k(u^1 - T_{m,\varepsilon}(u^n)\psi_j)T_{m,\varepsilon}(u^n)\diff x\diff\tau.
\end{align*}
Due to the presence of $\nabla_x \psi_j$ we know we can integrate only over $(0,t)\times (\Omega \setminus \Omega_j)$. In addition, it follows from the definition of $T_k$ that $|T_k'|\le 1$ and furthermore,  since $|\psi_j|\le 1$ we see that $T'_k(u^1 - T_{m,\varepsilon}(u^n)\psi_j)=0$ on the set, where $|u^1|>k+m+\varepsilon$. Therefore, we have the following estimate
\begin{align*}
    &\left|\int_0^t\int_{\Omega} (A(\tau, x, \nabla_x u^1)\cdot\nabla_x \psi_j) \,T'_k(u^1 - T_{m,\varepsilon}(u^n)\psi_j)T_{m,\varepsilon}(u^n)\diff x\diff\tau \right|\\
    &\;\leq \int_0^T\int_{\Omega\setminus\Omega_j}|A(t, x, \nabla_x T_{k+m+\varepsilon}(u^1))||\nabla_x\psi\, T_{m,\varepsilon}(u^n)| \diff x\diff t\\
   &\;\leq \int_0^T\int_{\Omega\setminus\Omega_j}|A(t, x, \nabla_x T_{k+m+\varepsilon}(u^1))|^{p'(t,x)}\diff x \diff t+\int_{\Omega_T}|\nabla_x\psi\, T_{m,\varepsilon}(u^n)|^{p(t,x)} \diff x\diff t,
\end{align*}
where we also used the Young inequality in the second estimate. Since $u^1$ is an entropy solution, we know that $|A(t, x, \nabla_x T_{k+m+\varepsilon}(u^1))|^{p'(t,x)}\in L^1(\Omega_T)$ for any $k,m,\varepsilon$. Further, $|\Omega\setminus\Omega_j|\to 0$ as $j\to \infty$, and therefore we see that the first integral tends to $0$ as $j\to\infty$. For the second integral, we use Lemma~\ref{proof:uniqueness:crucial_lemma} to see, that it vanishes in the limit as well. Hence,
\begin{align}\label{proof:uniqueness_psi_1}
    \lim_{j\to\infty}\int_0^t\int_{\Omega} (A(\tau, x, \nabla_x u^1)\cdot\nabla_x \psi_j) \,T'_k(u^1 - T_{m,\varepsilon}(u^n)\psi_j)T_{m,\varepsilon}(u^n)\diff x\diff\tau = 0.
\end{align}
In a similar manner, one may see that
\begin{align}
    \lim_{j\to\infty}\int_0^t\int_{\Omega}(A(\tau, x, \nabla_x T_{m,\varepsilon}(u^n))\cdot \nabla_x\psi_j) \,T_k(u^1 - T_{m,\varepsilon}(u^n)\psi_j)\diff x\diff \tau &= 0\label{proof:uniqueness_psi_2},\\
    \lim_{j\to\infty}\int_0^t\int_{\Omega}(A(\tau, x, \nabla_x T_{m,\varepsilon}(u^n))\cdot \nabla_x\psi_j)\, T'_k(u^1 - T_{m,\varepsilon}(u^n)\psi_j)\,T_{m,\varepsilon}(u^n)\,\psi_j\diff x\diff \tau &= 0 \label{proof:uniqueness_psi_3}.
\end{align}
Combining \eqref{proof:uniqueness_psi_1}, \eqref{proof:uniqueness_psi_2} and \eqref{proof:uniqueness_psi_3} we deduce from \eqref{proof:uniquenessthirdinequality}
\begin{equation}\label{proof:uniquenessfourthinequality}
    \begin{split}
        &\int_0^t\int_{\Omega} (A(\tau, x, \nabla_x u^1) - A(\tau, x, \nabla_x u^n)T'_{m,\varepsilon}(u^n))\cdot\nabla_x T_k(u^1 - T_{m,\varepsilon}(u^n))\diff x\diff \tau\\
        &\phantom{=} \int_{\Omega}G_k(u^1(t, x) - T_{m,\varepsilon}(u^n(t, x)))\diff x\\
        &\leq \int_0^t\int_{\Omega}(f - f^n\,T'_{m,\varepsilon}(u^n))\,T_k(u^1 - T_{m,\varepsilon}(u^n))\diff x\diff \tau+\int_{\Omega}G_k(u_0(x) - T_{m,\varepsilon}(u_0^n(x)))\diff x\\
        &\phantom{=}+\int_0^t\int_{\Omega}A(\tau, x,\nabla_x u^n)\cdot \nabla_x u^n \, T''_{m,\varepsilon}(u^n)(T_k(u^1 - T_{m,\varepsilon}(u^n)))\diff x\diff \tau.%\\
        %&\phantom{=}+\int_{\Omega}G_k(u_0(x) - T_{m,\varepsilon}(u_0^n(x)))\diff x.
    \end{split}
\end{equation}
%\\
\subsection{Step 4: convergence with \texorpdfstring{$\varepsilon\to 0_+$}{e}} Here, all the terms without the second derivative of $T_{m, \varepsilon}(u^n)$ can be dealt with using the Lebesgue dominated convergence theorem, thus we only focus on the term that does contain it. By the H\"{o}lder inequality, we have
\begin{equation}\label{proof:uniqueness_bound_on_second_derivative_1}
\begin{split}
    \int_0^t\int_{\Omega}A(\tau, x,\nabla_x u^n)\cdot\nabla_x u^n \, T''_{m,\varepsilon}(u^n)&(T_k(u^1 - T_{m,\varepsilon}(u^n)))\diff x\diff \tau\\
    &\leq k\int_0^t\int_{\Omega}(A(\tau, x,\nabla_x u^n)\cdot\nabla_x u^n)|T''_{m,\varepsilon}(u^n)|\diff x\diff \tau\\
    &= -k\int_0^t\int_{\Omega}(A(\tau, x,\nabla_x u^n)\cdot\nabla_x u^n)\,T''_{m,\varepsilon}(|u^n|)\diff x\diff \tau,
\end{split}
\end{equation}
where we used the monotonicity of $A$, i.e., \ref{intro:assumA_mono}--\ref{intro:assumA_vanish}, and the properties of $T_{k,\varepsilon}$. To estimate the right-hand side, we will mimic the computation in \eqref{abc}--\eqref{boundsthirdterm}, but instead of proving just boundedness (similarly as in \eqref{boundsthirdterm}), we need to proceed slightly differently and show that it is reasonably small.  Here, the key assumption that allow us to prove it is the $L^1$ integrability of data, i.e., \eqref{app:strongconvergenceoff}--\eqref{app:strongconvergenceofboundarydata}.

Since we did all rigorous step already in previous section (see the computation  \eqref{abc}--\eqref{boundsthirdterm}), we proceed here more formally. We set
$$
\phi = \mathcal{R}^\alpha(((1 - T'_{m,\varepsilon}((\overline{u^n})_+^\kappa))\psi)^\kappa\gamma^s_{-\eta, t}),
$$
in \eqref{eq:weak_equality_extensions}, where $u^n_+ = \max\{u^n, 0\}$. After converging with $\alpha\to 0_+$ and $s\to 0$ with the use of the Lebesgue dominated convergence theorem, one obtains
\begin{equation*}
    \begin{split}
        &-\int_{0}^t\int_{\Omega}(A(\tau, x, \nabla_x u^n)^\kappa\cdot\nabla_x{(u^n)^\kappa})T''_{m, \varepsilon}((u^n)^\kappa_+)\psi\diff x\diff\tau\\
        &\phantom{=}+ \int_0^t\int_{\Omega}(A(\tau, x, \nabla_x u^n)^\kappa \cdot\nabla_x \psi)(1 - T'_{m,\varepsilon}((u^n)_+^\kappa)\diff x\diff\tau\\
        &= -\int_{\Omega}((u^n)^\kappa_+ (t, x) - T_{m,\varepsilon}((u^n(t, x))^\kappa_+))\psi - ((u^n_0(x))^\kappa_+ - T_{m,\varepsilon}((u^n_0(x))^\kappa_+))\psi\diff x\\
        &\phantom{=}+ \int_0^t\int_{\Omega}f^n((1 - T'_{m,\varepsilon}((u^n)_+^\kappa))\psi)^\kappa\diff x\diff \tau.
    \end{split}
\end{equation*}
Here, for the left-hand side we may converge similarly as in \eqref{app:firstconvergencewithkappathirdterm} and \eqref{app:convergenceinkappathirdtermsecondone} and on the right-hand side, we simply use the almost everywhere convergence of mollification as well as the Lebesgue dominated convergence theorem to see
\begin{equation}\label{plus}
    \begin{split}
        &-\int_{0}^t\int_{\Omega}(A(\tau, x, \nabla_x u^n)\cdot\nabla_x u^n)T''_{m, \varepsilon}(u^n_+)\psi\diff x\diff\tau\\
        &\phantom{=}+ \int_0^t\int_{\Omega}(A(\tau, x, \nabla_x u^n)\cdot\nabla_x \psi)(1 - T'_{m,\varepsilon}(u^n_+))\diff x\diff\tau\\
        &= -\int_{\Omega}(u^n_+ (t, x) - T_{m,\varepsilon}(u^n_+(t, x)))\psi - ((u^n_0(x))_+ - T_{m,\varepsilon}((u^n_0(x))_+))\psi\diff x\\
        &\phantom{=}+ \int_0^t\int_{\Omega}f^n((1 - T'_{m,\varepsilon}(u^n_+))\psi)\diff x\diff \tau.
    \end{split}
\end{equation}
We can perform a very similar computation, but with $u^n_+$ replaced with $u^n_{-}:=\min\{0,u^n\}$, to get
\begin{equation}\label{minus}
    \begin{split}
        &-\int_{0}^t\int_{\Omega}(A(\tau, x, \nabla_x u^n)\cdot\nabla_x u^n)T''_{m, \varepsilon}(u^n_-)\psi\diff x\diff\tau\\
        &\phantom{=}+ \int_0^t\int_{\Omega}(A(\tau, x, \nabla_x u^n)\cdot\nabla_x \psi)(1 - T'_{m,\varepsilon}(u^n_-))\diff x\diff\tau\\
        &= -\int_{\Omega}(u^n_- (t, x) - T_{m,\varepsilon}(u^n_-(t, x)))\psi - ((u^n_0(x))_- - T_{m,\varepsilon}((u^n_0(x))_-))\psi\diff x\\
        &\phantom{=}+ \int_0^t\int_{\Omega}f^n((1 - T'_{m,\varepsilon}(u^n_-))\psi)\diff x\diff \tau.
    \end{split}
\end{equation}
Subtracting \eqref{minus} from \eqref{plus}, we get
\begin{equation}\label{rovna}
    \begin{split}
        &-\int_{0}^t\int_{\Omega}(A(\tau, x, \nabla_x u^n)\cdot\nabla_x u^n)T''_{m, \varepsilon}(|u^n|)\psi\diff x\diff\tau\\
        &= -\int_{\Omega}(|u^n (t, x)| - T_{m,\varepsilon}(|u^n(t, x)|))\psi - (|u^n_0(x)| - T_{m,\varepsilon}(|u^n_0(x)|))\psi\diff x\\
        &\phantom{=}- \int_0^t\int_{\Omega}(A(\tau, x, \nabla_x u^n)\cdot\nabla_x \psi)(T'_{m,\varepsilon}(u^n_-) - T'_{m,\varepsilon}(u^n_+))\diff x\diff\tau\\
        &\phantom{=}+ \int_0^t\int_{\Omega}f^n((T'_{m,\varepsilon}(u^n_-) - T'_{m,\varepsilon}(u^n_+))\psi)\diff x\diff \tau.
    \end{split}
\end{equation}
%
%\begin{equation}
%    \begin{split}
%        &-\int_{0}^t\int_{\Omega}(A(\tau, x, \nabla_x u^n)\cdot\nabla_x{u^n})T''_{m, \varepsilon}(|u^n|)\psi\diff x\diff\tau\\
 %       &\phantom{=}+ \int_0^t\int_{\Omega}(A(\tau, x, \nabla_x u^n) \cdot\nabla_x \psi)(1 - T'_{m,\varepsilon}(u^n))\diff x\diff\tau\\
 %       &\leq -\int_{\Omega}(u^n (t, x) - T_{m,\varepsilon}(u^n(t, x)))\psi - (u^n_0(x) - T_{m,\varepsilon}(u^n_0(x)))\psi\diff x\\
 %       &\phantom{=}+ \int_0^t\int_{\Omega}f^n(1 - T'_{m,\varepsilon}(u^n))\psi\diff x\diff \tau.
 %   \end{split}
%\end{equation}
After that, we converge with $\psi \to 1$. As the proof is similar to the one in the previous step for \eqref{proof:uniqueness_psi_1}, we skip the details here and we get
\begin{equation}\label{proof:uniqueness_bound_on_second_derivative}
    \begin{split}
        &-\int_{0}^t\int_{\Omega}(A(\tau, x, \nabla_x u^n)\cdot\nabla_x u^n)T''_{m, \varepsilon}(|u^n|)\diff x\diff\tau\\
        &\le -\int_{\Omega}(|u^n (t, x)| - T_{m,\varepsilon}(|u^n(t, x)|)) - (|u^n_0(x)| - T_{m,\varepsilon}(|u^n_0(x)|))\diff x\\
        %&\phantom{=}- \int_0^t\int_{\Omega}(A(\tau, x, \nabla_x u^n)\cdot\nabla_x \psi)(T'_{m,\varepsilon}(u^n_-) - T'_{m,\varepsilon}(u^n_+))\diff x\diff\tau\\
        &\phantom{=}+ \int_0^t\int_{\Omega}f^n(T'_{m,\varepsilon}(u^n_-) - T'_{m,\varepsilon}(u^n_+)))\diff x\diff \tau\\
        &\le \int_{\Omega}|u^n_0(x)|\chi_{\{|u^n_0|>m\}}\diff x +\int_{\Omega_T}|f^n|\chi_{\{|u^n|>m\}}\diff x\diff \tau.
    \end{split}
\end{equation}

Using  \eqref{proof:uniqueness_bound_on_second_derivative} in \eqref{proof:uniqueness_bound_on_second_derivative_1} and substituting the result into \eqref{proof:uniquenessfourthinequality}, we get
\begin{equation}\label{proof:uniquenessfifthinequality}
    \begin{split}
        &\int_0^t\int_{\Omega} (A(\tau, x, \nabla_x u^1) - A(\tau, x, \nabla_x u^n)T'_{m,\varepsilon}(u^n))\cdot\nabla_x T_k(u^1 - T_{m,\varepsilon}(u^n))\diff x\diff \tau\\
        &\phantom{=}+\int_{\Omega}G_k(u^1(t, x) - T_{m,\varepsilon}(u^n(t, x)))\diff x\\
        &\leq \int_0^t\int_{\Omega}(f - f^n\,T'_{m,\varepsilon}(u^n))\,T_k(u^1 - T_{m,\varepsilon}(u^n))\diff x\diff \tau\\
        &\phantom{=}+k \left(  \int_{\Omega}|u^n_0(x)|\chi_{\{|u^n_0|>m\}}\diff x +\int_{\Omega_T}|f^n|\chi_{\{|u^n|>m\}}\diff x\diff \tau \right)\\
        &\phantom{=}+ k\int_0^t\int_{\Omega}f^n(1 - T'_{m,\varepsilon}(u^n))\diff x\diff \tau +\int_{\Omega}G_k(u_0(x) - T_{m,\varepsilon}(u_0^n(x)))\diff x\\
        &\le C k \left(  \int_{\Omega}|u^n_0(x)|\chi_{\{|u^n_0|>m\}}\diff x +\int_{\Omega_T}|f^n|\chi_{\{|u^n|>m\}}\diff x\diff \tau +\int_{\Omega_T}|f^n-f|\diff x\diff \tau\right)\\
        &\phantom{=}+\int_{\Omega}G_k(u_0(x) - T_{m,\varepsilon}(u_0^n(x)))\diff x,
    \end{split}
\end{equation}
where we have used the triangle inequality and properties of functions $T_k$, and $T_{k,\varepsilon}$ respectively. %Notice that
%\begin{align}\label{ineq:fninequalityfirst}
%    \int_0^t\int_{\Omega}f^n(1 - T'_{m,\varepsilon}(u^n))\diff x\diff t\tau \leq \int_0^t\int_{\Omega}|f^n|(1 - T'_{m-1, 1}(u^n))\diff x\diff\tau,
%\end{align}
%and by triangle and H\"{o}lder's inequality
%\begin{align}\label{ineq:fninequalitysecond}
%    \int_0^t\int_{\Omega}(f - f^n\,T'_{m,\varepsilon}(u^n))\,T_k(u^1 - T_{m,\varepsilon}(u^n))\diff x\diff \tau \leq k\|f - f^n\|_1 + k\int_0^t\int_{\Omega}|f^n|(1 - T'_{m-1, 1}(u^n))\diff x\diff\tau.
%\end{align}
%Thus, with the use of \eqref{ineq:fninequalityfirst} and \eqref{ineq:fninequalitysecond} in \eqref{proof:uniquenessfifthinequality} we deduce
%\begin{equation}\label{proof:uniquenesssixthinequality}
%    \begin{split}
%        &\int_0^t\int_{\Omega} (A(\tau, x, \nabla_x u^1) - A(\tau, x, \nabla_x T_{m,\varepsilon}(u^n)))\cdot\nabla_x T_k(u^1 - T_{m,\varepsilon}(u^n))\diff x\diff \tau\\
%        &\leq k\|f - f^n\|_1 + 2k\int_0^t\int_{\Omega}|f^n|(1 - T'_{m - 1,1}(u^n))\diff x\diff \tau\\
%        &\phantom{=}-k\int_{\Omega}(u^n (t, x) - T_{m,\varepsilon}(u^n(t, x))) - (u^n_0(x) - T_{m,\varepsilon}(u^n_0(x)))\diff x\\
%        &\phantom{=}+\int_{\Omega}G_k(u_0(x) - T_{m,\varepsilon}(u_0^n(x)))\diff x - \int_{\Omega}G_k(u^1(t, x) - T_{m,\varepsilon}(u^n(t, x)))\diff x.
%    \end{split}
%\end{equation}
At this point, it is rather standard to let $\varepsilon \to 0_+$ and also $t\to T_{-}$, and observe
\begin{equation}\label{proof:uniquenessseventhinequality}
    \begin{split}
        &\int_{\Omega_T} (A(\tau, x, \nabla_x u^1) - A(\tau, x, \nabla_x T_{m}(u^n)))\cdot\nabla_x T_k(u^1 - T_{m}(u^n))\diff x\diff \tau\\
        &\phantom{=}+\int_{\Omega}G_k(u^1(t, x) - T_{m}(u^n(t, x)))\diff x\\
        &\le C k \left(  \int_{\Omega}|u^n_0(x)|\chi_{\{|u^n_0|>m\}}\diff x +\int_{\Omega_T}|f^n|\chi_{\{|u^n|>m\}}\diff x\diff \tau +\int_{\Omega_T}|f^n-f|\diff x\diff \tau\right)\\
        &\phantom{=}+\int_{\Omega}G_k(u_0(x) - T_{m}(u_0^n(x)))\diff x.
    \end{split}
\end{equation}

\subsection{Step 5: convergence with \texorpdfstring{$n\to\infty$}{n}} Our goal is to let $n\to \infty$ in \eqref{proof:uniquenessseventhinequality}. To proceed in the first term on the left-hand side, we apply the same procedure as in the proof of the inequality \eqref{existence_of_entropy_sol_inequality}. Indeed, using the definition of $T_k$, we have
\begin{equation*}%\label{proof:uniquenessseventhinequality}
    \begin{split}
        &\int_{\Omega_T} (A(\tau, x, \nabla_x u^1) - A(\tau, x, \nabla_x T_{m}(u^n)))\cdot\nabla_x T_k(u^1 - T_{m}(u^n))\diff x\diff \tau\\
        &\phantom{=}=\int_{\Omega_T} (A(\tau, x, \nabla_x T_{k+m}(u^1)) - A(\tau, x, \nabla_x T_{m}(u^n)))\cdot\nabla_x T_k(T_{k+m}(u^1) - T_{m}(u^n))\diff x\diff \tau
    \end{split}
\end{equation*}
and therefore it follows from \eqref{pointA} and \eqref{weaklimittrunaction}, that
\begin{equation*}%\label{proof:uniquenessseventhinequality}
    \begin{split}
        &\liminf_{n\to \infty} \int_{\Omega_T} (A(\tau, x, \nabla_x u^1) - A(\tau, x, \nabla_x T_{m}(u^n)))\cdot\nabla_x T_k(u^1 - T_{m}(u^n))\diff x\diff \tau\\
        &\phantom{=}\ge \int_{\Omega_T} (A(\tau, x, \nabla_x u^1) - A(\tau, x, \nabla_x T_{m}(u)))\cdot\nabla_x T_k(u^1 - T_{m}(u))\diff x\diff \tau.
    \end{split}
\end{equation*}
Next, for the second term on the left-hand side we can use the Fatou lemma. Finally, on the remaining terms on the right-hand side, we simply use  \eqref{app:strongconvergenceofboundarydata}, \eqref{app:strongconvergenceoff} and deduce the final inequality
\begin{equation}\label{proof:uniquenesseigthinequality}
    \begin{split}
        &\int_{\Omega_T} (A(\tau, x, \nabla_x u^1) - A(\tau, x, \nabla_x T_{m}(u)))\cdot\nabla_x T_k(u^1 - T_{m}(u))\diff x\diff \tau\\
        &\phantom{=}+\int_{\Omega}G_k(u^1(t, x) - T_{m}(u(t, x)))\diff x\\
        &\le C k \left(  \int_{\Omega}|u_0(x)|\chi_{\{|u_0|\ge m\}}\diff x +\int_{\Omega_T}|f|\chi_{\{|u|\ge m\}}\diff x\diff \tau \right)\\
        &\phantom{=}+\int_{\Omega}G_k(u_0(x) - T_{m}(u_0(x)))\diff x.
    \end{split}
\end{equation}

\subsection{Step 6: conclusion} Once again, we use the Fatou lemma to treat the left-hand side of the inequality and, remembering that $u$ is finite almost everywhere, we can converge with $m\to \infty$ to obtain from \eqref{proof:uniquenesseigthinequality} that for almost all $t\in (0,T)$
\begin{align*}
   \int_{\Omega}G_k(u^1(t, x) - u(t, x))\diff x + \int_{\Omega_T} (A(\tau, x, \nabla_x u^1) - A(\tau, x, \nabla_x u))\cdot\nabla_x T_k(u^1 - u)\diff x\diff \tau \leq 0.
\end{align*}
By the definition of $G_k$ and the monotonicity condition \ref{intro:assumA_mono} we know that
\begin{align*}
   \int_{\Omega}G_k(u^1(t, x) - u(t, x))\diff x + \int_0^t\int_{\Omega} (A(\tau, x, \nabla_x u^1) - A(\tau, x, \nabla_x u))\cdot\nabla_x T_k(u^1 - u)\diff x\diff \tau \geq 0.
\end{align*}
Consequently, letting also $k\to \infty$ we deduce that $u^1=u$ almost everywhere in $\Omega_T$.
%\end{proof}

\appendix
\section{Musielak-Orlicz spaces \texorpdfstring{$L^{p(t, x)}(\Omega_T)$}{Ls}}\label{app:musielaki}

\noindent Here, we remark some of the basic properties of the variable exponent spaces $L^{p(t, x)}(\Omega_T)$. For  more details about the Musielak--Orlicz spaces, we refer the interested reader to~\cite{chlebicka2019book}. We start with the definition

\begin{Def}\label{def:musielak_exponent_space}
Given a measurable function $p(t, x): \Omega_T \to [1,\infty)$, we let
\begin{align*}
L^{p(t, x)}(\Omega_T) = \left\{ \xi: \Omega_T \to \mathbb{R}^d: \mbox{ there is } \lambda>0 \mbox{ such that } \int_{\Omega_T} \left| \frac{\xi(t,x)}{\lambda}\right|^{p(t, x)} \diff x \diff t < \infty \right\}.
\end{align*}
Moreover, if $p(t, x)$ satisfies the boundedness conditions \ref{ass:usual_bounds_exp} or \ref{ass:usual_bounds_exp_entr} then this definition is equivalent to
$$
L^{p(t, x)}(\Omega_T) = \left\{ \xi: \Omega_T \to \mathbb{R}^d:  \int_{\Omega_T} \left| {\xi(t,x)}\right|^{p(t, x)} \diff x \diff t < \infty \right\}.
$$
\end{Def}

\noindent Variable exponent spaces are Banach with the norm below.

\begin{thm}
Let
\begin{equation}\label{intro:norm}
\| \xi \|_{L^{p(t, x)}} = \inf \left\{\lambda>0: \int_{\Omega_T} \left| \frac{\xi(t,x)}{\lambda}\right|^{p(t, x)} \diff x \diff t \leq 1 \right\}.
\end{equation}
Then, $\left(L^{p(t, x)}, \|\cdot\|_{L^{p(t, x)}}\right)$ is a Banach space.
\end{thm}

\noindent We will be interested in the two kinds of convergences in the aforementioned spaces.

\begin{Def}
We say that $\xi_n$ converges strongly to $\xi$ (denoted $\xi_n \to \xi$) in $L^{p(t, x)}(\Omega_T)$, if
$$
\|\xi_n - \xi\|_{L^{p(t, x)}}\rightarrow 0,
$$
and that $\xi_n$ converges modularly to $\xi$, if there exists $\lambda > 0$ such that
$$
\int_{\Omega_T} \left| \frac{\xi_n(t,x) - \xi(t,x)}{\lambda} \right|^{p(t, x)} \diff x \diff t \to 0.
$$
\end{Def}

\noindent Modular and strong convergences are connected in the following form.

\begin{thm} The following equivalence holds true:
\begin{align*}
    \|\xi_n - \xi\|_{L^{p(t, x)}} \to 0 \Longleftrightarrow \int_{\Omega_T} \left| \frac{\xi_n(t,x) - \xi(t,x)}{\lambda} \right|^{p(t, x)} \diff x \diff t \to 0 \text{ for every }\lambda >0
\end{align*}
\end{thm}

\begin{cor}
If $p(t, x)$ satisfies the boundedness conditions \ref{ass:usual_bounds_exp} or \ref{ass:usual_bounds_exp_entr}, then strong and modular convergences are equivalent.
\end{cor}

\noindent In the case of the variable exponent $L^{p(t, x)}$ spaces, we have a version of the H\"{o}lder inequality.
\begin{prop}\label{prop:generalisedholder}
Suppose $f\in L^{p(t, x)}(\Omega)$ and $g\in L^{p'(t, x)}(\Omega)$. Then,
$$
\int_{\Omega}|f\,g|\diff x \leq 2\|f\|_{L^{p(t, x)}}\|g\|_{L^{p'(t, x)}}.
$$
\end{prop}

\begin{prop}\label{prop:convergence_of_convolution_in_musielak_orlicz}
Let $f\in L^{p(t, x)}(\Omega_T)$ with $p$ satisfying Assumption \ref{ass:exponent_cont_space} or Assumption \ref{ass:exponent_cont_space_entr}, then for any $\psi\in C^\infty_c(\Omega)$
$$
f^\kappa\psi \rightarrow f\psi\text{, in }L^{p(t, x)}(\Omega_T).
$$
\end{prop}

\noindent The theorem below makes a connection between modular convergence and the convergences of the products as formulated below.

\begin{thm}\label{thm:modular_l1}\textbf{\textup{(Proposition 2.2, \cite{gwiazda2008onnonnewtonian})}}
Assume that the function $p(t, x)$ satisfies Assumption \ref{ass:exponent_cont_space} or Assumption \ref{ass:exponent_cont_space_entr}. Presuppose also that $\phi_n \to \phi$ modularly in $L^{p(t, x)}(\Omega_T)$ and $\psi_n\to \psi$ modularly in $L^{p'(t, x)}(\Omega_T)$. Then, $\phi_n\,\psi_n \to \phi\,\psi$ in $L^1(\Omega_T)$.
\end{thm}

\section{Useful results}\label{Ap2}
\begin{lem}\label{aubin-lions}\textup{\textbf{(Generalized Aubin--Lions lemma, \cite[Lemma 7.7]{MR3014456})}}
Denote
$$
W^{1, p, q}(I; X_1, X_2) := \left\{u\in L^p(I; X_1); \frac{du}{dt}\in L^q(I; X_2)\right\}.
$$
Then, if $X_1$ is a separable, reflexive Banach space, $X_2$ is a Banach space and $X_3$ is a metrizable locally convex Hausdorff space, $X_1$ embeds compactly into $X_2$, $X_2$ embeds continuously into $X_3$, $1 < p <\infty$ and $1\leq q\leq \infty$, we have
$$
W^{1, p, q}(I; X_1, X_3) \text{ embeds compactly into }L^p(I; X_2).
$$
In particular, any bounded sequence in $W^{1, p, q}(I; X_1, X_3)$ has a convergent subsequence in $L^p(I; X_2)$.
\end{lem}

\begin{lem}\label{res:monot_trick}\textup{\textbf{(Lemma 2.16, \cite{bulicek2021parabolic})}}
Let the mapping $A$ satisfy Assumption~\ref{intro:ass_on_A}. Assume there are $\chi \in L^{p'(t,x)}(\Omega_T;\R^d)$ and $\xi \in L^{p(t,x)}(\Omega_T;\R^d)$, such that
\begin{equation*}
\int_{\Omega_T} \left(\chi - A(t,x,\eta)\right) \cdot (\xi - \eta)\, \psi(x)\diff t \diff x \geq 0
\end{equation*}
for all $\eta \in L^{\infty}(\Omega_T; \R^d)$, and $\psi \in C_0^{\infty}(\Omega)$ with $0 \leq \psi \leq 1$. Then,
$$
A(t,x,\xi) = \chi(t,x) \mbox{ a.e. in } \Omega_T.
$$
\end{lem}

\begin{lem}\label{proof:uniqueness:crucial_lemma}
Let $\Omega$ be a Lipschitz domain. Suppose that the function $p$ satisfies Assumption \ref{ass:exponent_cont_space} or Assumption \ref{ass:exponent_cont_space_entr}. Then, there is a family of functions $\left\{\psi_j \right\}_{j \in \N}$ and a family of open sets $\Omega_j \subset \Omega$ fulfilling
\begin{itemize}
    \item $\psi_j\in C^\infty_c(\Omega)$, with support $\Omega_j$,
    \item $|\Omega\setminus\Omega_j|\to 0$ as $j\to\infty$,
    \item $0\leq \psi_j\leq 1$,
    \item $\psi_j \to 1$ as $j\to \infty$,
    \item $\nabla_x\psi_j = 0$ on $\Omega_j$,
\end{itemize}
such that if $u \in L^{\infty}(0,T; L^1(\Omega)) \cap L^1(0,T; W^{1,1}_0(\Omega))$ with $\nabla_x u \in L^{p(t, x)}(\Omega_T;\R^d)$, we have
$$
\int_0^T \int_{\Omega} \left|\nabla\psi_{j}(x) \, u(t,x)\right|^{p(t, x)} \,\diff x \diff t \to 0 \mbox{ as } j \to \infty.
$$
\end{lem}
\begin{proof}
    As the proof is very similar to \cite[Lemma 4.1]{bulicek2021parabolic}, we will skip some of the technical details concerning the assumption on $\Omega$ being a Lipschitz domain. Define $\Omega_j = \left\{x\in\Omega:\, \mathrm{dist}(x, \p \Omega) > \frac{1}{j}\right\}$, so that $|\Omega\setminus\Omega_j|\to 0$ as $j\to \infty$. Moreover, let $\psi_j\in C^\infty_c(\Omega)$ such that $\psi_j\equiv 1$ on $\Omega_j$, $\psi_j\equiv 0$ outside of some open cover $U$ of $\overline{\Omega}_j$ and $0 < \psi_j < 1$ in between those sets. Notice that $\nabla_x\psi_j \equiv 0$ on $\Omega_j$ and $|\nabla_x\psi_j|\leq Cj$, for some general constant $C > 0$. We cover the boundary layer $\Omega\setminus\Omega_j$ with the family of cubes $\{Q_m^j\}_{m = 1}^{N_j}$ with edge of length $\frac{1}{j}$. Denote
    $$
        q_m^j(t) = \inf_{x\in Q_m^j}p(t, x), \quad r_m^j(t) = \sup_{x\in Q_m^j}p(t, x).
    $$
    By the log-H\"{o}lder continuity assumptions \ref{ass:cont} or \ref{ass:cont_entr}, as well as boundedness \ref{ass:usual_bounds_exp}, \ref{ass:usual_bounds_exp_entr} we may fix $j$ big enough, so that
    \begin{align}\label{app:some_ineq_on_local_exponent}
    r_m^j(t) - q_m^j(t) \leq \frac{1}{d + p_{\mathrm{max}}}.
    \end{align}
    Then,
    \begin{align}
        & r_m^j(t) - q_m^j(t) < r_m^j(t) - r_m^j(t)\frac{d}{d+ r_m^j(t)} \Longleftrightarrow \frac{q_m^j(t)d}{d - q_m^j(t)} > r_m^j(t)\,\, \text{ (in case $q_m^j(t) < d$)},\label{app:sobolev_embed_ineq}\\
        & r_m^j(t) < q_m^j(t)\left(1 + \frac{1}{d}\right)\label{app:some_ineq_on_local_exponent_2}.
    \end{align}
    Notice, that \eqref{app:sobolev_embed_ineq} in particular implies
    \begin{align}\label{app:sobolev_embedding}
        W^{1, q_m^j(t)}(Q_m^j) \hookrightarrow L^{r_m^j(t)}(Q_m^j).
    \end{align}
    Here, let us define $v(x) = u\left(\frac{x}{j}\right)$. Before moving forward, let us notice that as $\Omega$ is Lipschitz, we can assume without loss of generality, that its boundary is contained in the plane $\{x\in \R^d:\, x_d = 0\}$. Then, using absolute continuity on lines for Sobolev maps \cite[Theorem 4.21]{evans2015measure}, we can write
    $$
    v(s, x) = \int_0^{x_d}\p_{x_d}v(s, x_1, ..., x_{d-1}, r)\diff r.
    $$
    Hence, by Jensen's inequality
    \begin{equation}\label{app:gradient_ineq_for_sobolev}
    \begin{split}
    \int_{Q^1_m}|v(s, x)|^{q_m^j(s)}\diff x &\leq \int_{Q^1_m}\int_0^{x_d}|\p_{x_d}v(s, x_1, ..., x_{d-1}, r)|^{q_m^j(s)}\diff r\diff x \\
    &\leq \int_0^1\int_{Q^1_m}|\nabla_x v(s, x_1, ..., x_{d-1}, r)|^{q_m^j(s)}\diff x\diff r\\
    &\phantom{=}= \int_{Q^1_m}|\nabla_x v(s, x)|^{q_m^j(s)}\diff x.
    \end{split}
    \end{equation}
    Now, with the use of \eqref{app:sobolev_embedding} and \eqref{app:gradient_ineq_for_sobolev} we may imply
    \begin{equation}\label{app:interpolation_ineq_1}
        \begin{split}
            &\int_{Q_m^j}|u(s, x)|^{r_m^j(s)}\diff x = \frac{1}{j^d} \int_{Q_m^1}|v(x)|^{r_m^j(s)}\diff x \leq \frac{C}{j^d}\|v\|_{W^{1, q_m^j(t)}}^{r_m^j(s)} \\
            &\phantom{=}\leq \frac{C}{j^d}\left(\int_{Q_m^1}|\nabla_x v(s, x)|^{q_m^j(s)}\diff x\right)^{\frac{r_m^j(s)}{q_m^j(s)}} =\frac{C}{j^{d + r_m^j(s) - \frac{dr_m^j(s)}{q_m^j(s)}}} \left(\int_{Q_m^j}|\nabla_x u(s, x)|^{q_m^j(s)}\diff x\right)^{\frac{r_m^j(s)}{q_m^j(s)}},
        \end{split}
    \end{equation}
    as well as for any $\alpha\in (0, 1)$ (by \eqref{app:sobolev_embedding}, \eqref{app:gradient_ineq_for_sobolev} and Littlewood's interpolation inequality)
    \begin{equation*}
        \begin{split}
            &\int_{Q_m^j}|u(s, x)|^{r_m^j(s)}\diff x = \frac{1}{j^d} \int_{Q_m^1}|v(x)|^{r_m^j(s)}\diff x\\
            &\leq \frac{C}{j^d}\left(\int_{Q_m^1}|\nabla_x v(s, x)|^{q_m^j(s)}\diff x\right)^{\frac{r_m^j(s)}{q_m^j(s)}\alpha}\left(\int_{Q_m^1}|v(s, x)|\diff x\right)^{r_m^j(s)(1 - \alpha)}\\
            &=\frac{C}{j^{d + \alpha\left(r_m^j(s) - \frac{dr_m^j(s)}{q_m^j(s)}\right) + (1-\alpha)r_m^j(s)d}} \left(\int_{Q_m^j}|\nabla_x u(s, x)|^{q_m^j(s)}\diff x\right)^{\frac{r_m^j(s)}{q_m^j(s)}\alpha}\left(\int_{Q_m^j}|u(s, x)|\diff x\right)^{r_m^j(s)(1 - \alpha)}.
        \end{split}
    \end{equation*}
    Setting
    $$
    \alpha = \frac{d\,q_m^j(s)(r_m^j(s) - 1)}{d\,r_m^j(s)(q_m^j(s) - 1) + r_m^j(s)q_m^j(s)}
    $$
    (which is between $0$ and $1$ by \ref{ass:usual_bounds_exp} or \ref{ass:usual_bounds_exp_entr} and \eqref{app:some_ineq_on_local_exponent}) we obtain
    \begin{align}\label{app:interpolation_ineq_2}
        \int_{Q_m^j}|u(s, x)|^{r_m^j(s)}\diff x \leq C\left(\int_{Q_m^j}|\nabla_x u(s, x)|^{q_m^j(s)}\diff x\right)^{\frac{r_m^j(s)}{q_m^j(s)}\alpha}\left(\int_{Q_m^j}|u(s, x)|\diff x\right)^{r_m^j(s)(1 - \alpha)}.
    \end{align}
    Moreover, by \eqref{app:some_ineq_on_local_exponent_2}
    $$
    \alpha < \frac{q_m^j(s)}{r_m^j(s)}.
    $$
    Thus, we can set $\lambda\in (0, 1)$ as
    $$
    \lambda = \frac{\frac{q_m^j(s)}{r_m^j(s)}-\alpha}{1 - \alpha}.
    $$
    Hence, by \eqref{app:interpolation_ineq_1} and \eqref{app:interpolation_ineq_2}
    \begin{equation}\label{app:inequality_for_lemma_1}
        \begin{split}
            &\int_{Q_m^j}|u(s, x)|^{r_m^j(s)}\diff x = \left(\int_{Q_m^j}|u(s, x)|^{r_m^j(s)}\diff x\right)^{\lambda}\left(\int_{Q_m^j}|u(s, x)|^{r_m^j(s)}\diff x\right)^{1 - \lambda}\\
            &\leq \frac{C}{j^{\lambda\left(d + r_m^j(s) - \frac{dr_m^j(s)}{q_m^j(s)}\right)}} \left(\int_{Q_m^j}|\nabla_x u(s, x)|^{q_m^j(s)}\diff x\right)^{\frac{r_m^j(s)}{q_m^j(s)}\lambda}\\
            &\phantom{=}\cdot\left(\int_{Q_m^j}|\nabla_x u(s, x)|^{q_m^j(s)}\diff x\right)^{\frac{r_m^j(s)}{q_m^j(s)}\alpha(1-\lambda)}\left(\int_{Q_m^j}|u(s, x)|\diff x\right)^{r_m^j(s)(1 - \alpha)(1-\lambda)}.
        \end{split}
    \end{equation}
    By the choice of $\lambda$
    $$
    \frac{r_m^j(s)}{q_m^j(s)}\lambda + \frac{r_m^j(s)}{q_m^j(s)}\alpha(1-\lambda) = 1,
    $$
    meaning that \eqref{app:inequality_for_lemma_1} reduces to
    \begin{multline*}
            \int_{Q_m^j}|u(s, x)|^{r_m^j(s)}\diff x \\
            \leq \frac{C}{j^{\lambda\left(d + r_m^j(s) - \frac{dr_m^j(s)}{q_m^j(s)}\right)}} \left(\int_{Q_m^j}|\nabla_x u(s, x)|^{q_m^j(s)}\diff x\right)\left(\int_{Q_m^j}|u(s, x)|\diff x\right)^{r_m^j(s)(1 - \alpha)(1-\lambda)},
    \end{multline*}
    which in turn implies
    \begin{equation}\label{ineqyality_for_lemma_3}
    \begin{split}
        &\int_{Q_m^j}|j\,u(s, x)|^{r_m^j(s)}\diff x \\
        &\leq \frac{C}{j^{\lambda\,d\left(\frac{q_m^j(s) - r_m^j(s)}{q_m^j(s)}\right) + (\lambda - 1)r_m^j(s)}} \left(\int_{Q_m^j}|\nabla_x u(s, x)|^{q_m^j(s)}\diff x\right)\left(\int_{Q_m^j}|u(s, x)|\diff x\right)^{r_m^j(s)(1 - \alpha)(1-\lambda)}\\
        &\phantom{=}= \frac{C}{j^{\lambda\,d\left(\frac{q_m^j(s) - r_m^j(s)}{q_m^j(s)}\right) + (\lambda - 1)r_m^j(s)}}\|u\|^{r_m^j(s) - q_m^j(s)}_{L^1(Q_m^j)} \left(\int_{Q_m^j}|\nabla_x u(s, x)|^{q_m^j(s)}\diff x\right).
    \end{split}
    \end{equation}
    Now, notice that by \eqref{app:some_ineq_on_local_exponent_2}
    $$
    \alpha < \frac{q_m^j(s)}{r_m^j(s)} < \frac{d}{d+1} \Longrightarrow \frac{1}{1 - \alpha} < d + 1,
    $$
    which together with $\lambda < 1$ and \ref{ass:usual_bounds_exp} or \ref{ass:usual_bounds_exp_entr} implies
    \begin{equation}\label{app:exponent_ineq_for_j}
        \begin{split}
            \lambda\,d\left(\frac{q_m^j(s) - r_m^j(s)}{q_m^j(s)}\right) + (\lambda - 1)r_m^j(s) &= (q_m^j(s) - r_m^j(s))\left(\frac{\lambda\,d}{m} + \frac{1}{1-\alpha}\right) \\
            &\geq -\left(\frac{d}{p_{\mathrm{max}}} + d + 1\right)(r_m^j(s) - q_m^j(s)).
        \end{split}
    \end{equation}
    Combining \eqref{ineqyality_for_lemma_3} and \eqref{app:exponent_ineq_for_j} we obtain
    \begin{equation}\label{app:inequality_for_lemma_4}
        \begin{split}
            &\int_{Q_m^j}|j\,u(s, x)|^{r_m^j(s)}\diff x \\
            &\leq \frac{C}{j^{\lambda\,d\left(\frac{q_m^j(s) - r_m^j(s)}{q_m^j(s)}\right) + (\lambda - 1)r_m^j(s)}}\|u\|^{r_m^j(s) - q_m^j(s)}_{L^1(Q_m^j)} \left(\int_{Q_m^j}|\nabla_x u(s, x)|^{q_m^j(s)}\diff x\right)\\
            &\leq C\,j^{\left(\frac{d}{p_{\mathrm{max}}} + d + 1\right)(r_m^j(s) - q_m^j(s))}\|u\|^{r_m^j(s) - q_m^j(s)}_{L^1(Q_m^j)} \left(\int_{Q_m^j}|\nabla_x u(s, x)|^{q_m^j(s)}\diff x\right).
        \end{split}
    \end{equation}
    By \ref{ass:cont} or \ref{ass:cont_entr} we can further estimate
    \begin{equation*}
        \begin{split}
            j^{\left(\frac{d}{p_{\mathrm{max}}} + d + 1\right)(r_m^j(s) - q_m^j(s))} = e^{\left(\frac{d}{p_{\mathrm{max}}} + d + 1\right)(r_m^j(s) - q_m^j(s))\ln(j)} \leq e^{\left(\frac{d}{p_{\mathrm{max}}} + d + 1\right)C\ln(j)^{-1}\,\ln(j)} \leq C'.
        \end{split}
    \end{equation*}
    Putting it into \eqref{app:inequality_for_lemma_4} we can see
    \begin{align}\label{app:inequality_for_lemma_5}
        \int_{Q_m^j}|j\,u(s, x)|^{r_m^j(s)}\diff x \leq C\|u\|^{r_m^j(s) - q_m^j(s)}_{L^1(Q_m^j)} \left(\int_{Q_m^j}|\nabla_x u(s, x)|^{q_m^j(s)}\diff x\right).
    \end{align}
    Finally, with \eqref{app:inequality_for_lemma_5} we can show the thesis of the lemma. We have
    \begin{equation}\label{app:inequality_for_lemma_6}
        \begin{split}
            &\int_0^t \int_{\Omega} \left|\nabla\psi_{j}(x) \, u(s,x)\right|^{p(s, x)} \,\diff x \diff s = \int_0^t \int_{\Omega\setminus\Omega_j} \left|\nabla\psi_{j}(x) \, u(s,x)\right|^{p(s, x)} \,\diff x \diff s \\
            &\lesssim \int_0^t \int_{\Omega\setminus\Omega_j} \left|j \, u(s,x)\right|^{p(s, x)} \,\diff x \diff s = \int_0^t\sum_{m = 1}^{N_j}\int_{Q_m^j}\left|j \, u(s,x)\right|^{p(s, x)} \,\diff x \diff s \\
            &\leq \int_0^t\sum_{m = 1}^{N_j}\int_{Q_m^j}1 + \left|j \, u(s,x)\right|^{r_m^j(s)} \,\diff x \diff s \\
            &\leq T|\Omega\setminus\Omega_j| + \int_0^t\sum_{m = 1}^{N_j}C\|u\|^{r_m^j(s) - q_m^j(s)}_{L^1(Q_m^j)} \left(\int_{Q_m^j}|\nabla_x u(s, x)|^{q_m^j(s)}\diff x\right) \diff s\\
            &\leq T|\Omega\setminus\Omega_j| + C(d, p_\mathrm{min},p_\mathrm{max}, \|u\|_{L^\infty_tL^1_x}) \int_0^t\sum_{m = 1}^{N_j}\int_{Q_m^j}|\nabla_x u(s, x)|^{q_m^j(s)}\diff x\diff s\\
            &\leq T|\Omega\setminus\Omega_j| + C(d, p_\mathrm{min},p_\mathrm{max}, \|u\|_{L^\infty_tL^1_x}) \int_0^t\sum_{m = 1}^{N_j}\int_{Q_m^j}1 + |\nabla_x u(s, x)|^{p(s, x)}\diff x\diff s\\
            &\leq C(d, p_\mathrm{min},p_\mathrm{max}, \|u\|_{L^\infty_tL^1_x})\left(T|\Omega\setminus\Omega_j| + \int_0^t\int_{\Omega\setminus\Omega_j}|\nabla_x u(s, x)|^{p(s, x)}\diff x\diff s\right).
        \end{split}
    \end{equation}
    From our assumptions $|\nabla_x u(s,x)|^{p(s, x)}\in L^1((0, T)\times \Omega)$, hence the right-hand side of \eqref{app:inequality_for_lemma_6} converges to $0$, and this ends the proof of the lemma.
\end{proof}

% \begin{prop}\label{res:approx_theorem}\textup{\textbf{(Theorem 2.4, \cite{bulicek2021parabolic})}}
% Let $\Omega \subset \R^d$, and $\psi:\Omega \to \R$ be arbitrary such that $\psi \in C_c^{\infty}(\Omega)$. Suppose that $p$ satisfies Assumption \ref{ass:exponent_cont_space} and let $u$ satisfy $u \in L^{\infty}((0, T); L^2(\Omega)) \cap L^1((0, T); W^{1,1}_{0}(\Omega))$ and $\nabla_x u \in L^{p(t,x)}(\Omega_T)$. If we extend $u(t,x)$ by zero for all $x\notin \Omega$, then there exists $\kappa_0>0$ depending on $\psi$ and $\Omega$ such that:
% \begin{enumerate}[label=(S\arabic*)]
% \item\label{thmitem:conv1} $\left(u^{\kappa} \psi\right)^{\kappa} \in L^{\infty}(0,T;C_0^{\infty}(\Omega))$ for all $\kappa \in (0, \kappa_0)$,
% \item\label{thmitem:conv2} $\left(u^{\kappa} \psi\right)^{\kappa}  \to u\, \psi$ a.e. in $\Omega_T$ and in $L^1(0,T; L^1(\Omega))$ as $\kappa \to 0^+$,
% \item\label{thmitem:conv3}
% $\nabla_x\left(u^{\kappa} \psi\right)^{\kappa} \in L^{p(t,x)}(\Omega_T)$ and
% $\nabla_x\left(u^{\kappa} \psi\right)^{\kappa} \to \nabla_x \left(u \psi\right)$ in $L^{p(t,x)}(\Omega_T)$ as $\kappa \to 0^+$.
% \end{enumerate}
% \end{prop}

\bibliographystyle{abbrv}
\bibliography{parabolic_l1}
\end{document}